\newif\ifArXiV

\ArXiVtrue	

\documentclass[a4paper]{article}
\usepackage[affil-it]{authblk}
\usepackage[authoryear]{natbib}
\newenvironment{frontmatter}{}{}

\let\address\affil
\newenvironment{keyword}{\small \textbf{Keywords: }}{}

\usepackage{fullpage}
\usepackage{lineno}   
\usepackage{graphics}  				
\usepackage{graphicx}  				
\usepackage{caption}
\usepackage{subcaption}
\usepackage{xspace}
\usepackage{hyperref}
\usepackage{natbib}

\usepackage{float}     				
\usepackage{epsfig}   				
\usepackage{amssymb} 
\usepackage{amsfonts} 
\usepackage{amsmath}
\usepackage{amsthm}

\newtheorem{definition}{Definition}

\newtheorem{proposition}{Proposition}

\newcommand{\mytag}[1]{(\hypertarget{#1}{#1})}
\newcommand{\myref}[1]{\textnormal{(\hyperlink{#1}{#1})}}

\usepackage[export]{adjustbox}

\usepackage[usenames, dvipsnames]{color}
\usepackage{soul}
\definecolor{gris}{gray}{0.5} 

\newcommand{\OP}{OP\xspace}
\newcommand{\OIG}{OIG\xspace}
\newcommand{\OIGB}{OIGB\xspace}
\newcommand{\OIGS}{OIGS\xspace}
\newcommand{\IG}{IG\xspace}
\newcommand{\IGs}{IGs\xspace}

\usepackage[english]{babel}

\usepackage{textcomp}
\usepackage{footnote}
\usepackage{pdflscape}                        				 
\usepackage{latexsym}
\usepackage{url}
\usepackage[ruled,linesnumbered,english]{algorithm2e} 
\SetKwRepeat{Do}{do}{while} 																									  
\usepackage{longtable,multirow,booktabs} 																		
\usepackage[utf8]{inputenc}
\usepackage{mathtools}
\renewcommand{\paragraph}[1]{{\vskip2mm\noindent \bf #1}\hspace{0.3cm}}
\begin{document}
\begin{frontmatter}

\title{Competing for the most profitable tour: The orienteering interdiction game}
\author[1,2]{Eduardo \'Alvarez-Miranda\thanks{ealvarez@utalca.cl}}
\author[3]{Markus Sinnl\thanks{markus.sinnl@jku.at}}
\author[4]{K\"ubra Tan{\i}nm{\i}\c{s}\thanks{ktaninmis@ku.edu.tr}}
\date{}

\address[1]{Department of Industrial Engineering, Faculty of Engineering, Universidad de Talca, Sede Curic\'o, Chile}
\address[2]{Instituto Sistemas Complejos de Ingenier\'{\i}a, Chile}
\address[3]{Institute of Business Analytics and Technology Transformation/JKU Business School, Johannes Kepler University Linz, 4040 Linz, Austria}
\address[4]{Department of Industrial Engineering, Ko\c{c} University, 34742 \.{I}stanbul, Turkey}

\maketitle	

\begin{abstract}
The orienteering problem is a well-studied and fundamental problem in transportation science. In the problem, we are given a graph with prizes on the nodes and lengths on the edges, together with a budget on the overall tour length. The goal is to find a tour that respects the length budget and maximizes the collected prizes. In this work, we introduce the orienteering interdiction game, in which a competitor (the leader) tries to minimize the total prize that the follower can collect within a feasible tour. To this end, the leader interdicts some of the nodes so that the follower cannot collect their prizes. 
The resulting interdiction game is formulated as a bilevel optimization problem, and a single-level reformulation is obtained based on interdiction cuts. A branch-and-cut algorithm with several enhancements, including the use of a solution pool, a cut pool and a heuristic method for the follower's problem, is proposed. In addition to this exact approach, a genetic algorithm is developed to obtain high-quality solutions in a short computing time. In a computational study based on instances from the literature for the orienteering problem, the usefulness of the proposed algorithmic components is assessed, and the branch-and-cut and genetic algorithms are compared in terms of solution time and quality.

\end{abstract}


\end{frontmatter}

\section{Introduction and problem definition}
\label{sec:intro}

In recent years, \emph{interdiction games (\IGs)} have received considerable attention in the logistics literature, see, e.g., the surveys \citep{smith2020survey}, \citep[Chapter 6]{kleinert2021survey} and \citep[Chapter 4.2]{beck2023survey}. \IGs involve two decision makers, usually called the \emph{leader} and the \emph{follower} who compete in a hierarchical manner. The follower solves an optimization problem defined over a set of \emph{assets} such as facilities or arcs on a network that the leader can interdict within an \emph{interdiction budget}. Depending on the concrete setting of the \IG, interdicting an asset can either deprecate its value for the follower or completely destroy the asset making it unusable for the follower. The goal of the leader is to choose the assets to interdict in such a way as to maximize the deterioration of the follower's objective function value. Interdiction games find applications in various areas such as marketing \citep{denegre2011interdiction}, identifying and defending critical infrastructure \citep{church2004identifying, brown2006defending}, as well as conservation planning \citep{sefair2017defender}. 
Network interdiction is an important class of \IGs, involving the interdiction of some network components at the upper level. Early examples in this area include the interdiction of flows on arcs \citep{wollmer1964removing, wood1993deterministic}, and the interdiction of shortest paths \citep{israeli2002shortest}. More recent works address, for example, multi-commodity flow interdiction \citep{lim2007algorithms}, traveling salesman problem with interdiction \citep{lozano2017solving}, and maximum clique interdiction \citep{furini2019maximum}. \cite{smith2020survey} provides a comprehensive survey on network interdiction problems.

In this work, we consider an interdiction version of the well-known \emph{orienteering problem} (\OP).
In the \OP we are given a graph with prizes on the nodes and lengths on the edges, along with a budget for the overall length of the tour and a depot node.
The goal is to find a tour that respects the length budget, passes through the depot, and maximizes the prizes collected (a node prize is collected if the node is part of the tour) \citep{golden1987orienteering}. As a result, the \OP is a combination of the knapsack problem \citep{dantzig1957discrete} and the traveling salesperson problem \citep{dantzig1954solution}, and it is also known as the selective traveling salesperson problem \citep{laporte1990selective}. This fundamental problem in logistics and transportation is NP-hard and has spawned countless variants and generalizations such as the team \OP \citep{chao1996team}, \OP with time windows \citep{labadie2012team}, and the stochastic \OP \citep{ilhan2008orienteering}. For an overview on the \OP, we refer to the surveys \citep{vansteenwegen2011orienteering,gunawan2016orienteering}.

In the interdiction version of the \OP, which we call the \emph{orienteering interdiction game} (\OIG), initially the leader interdicts some of the nodes within a budget. Then, the follower solves an \OP where it is not possible to collect prizes from the interdicted nodes. The aim of the leader is to choose the nodes to interdict in such a way that the maximum possible prize that the follower collects is minimized. 
A formal definition of \OIG is given in the following.

\begin{definition}[Orienteering interdiction game (\OIG)]
We are given an undirected complete graph $G=(V,E)$ with node set $V$, an edge set $E$, a prize $p_i>0$ associated with each node $i\in V$, and a length $d_{e}$ associated with each edge $e\in E$. 
When a node $i\in V$ is interdicted by the leader, its prize is captured by the leader and cannot be collected by the follower even if node $i$ is visited in the follower tour.
Therefore, the leader's goal is to interdict a subset of nodes, whose total cost does not exceed an interdiction budget ($Q_\ell$), so it minimizes the maximum profit (that is, the sum of the prizes of the nodes in the tour) that the follower can achieve by performing a tour, starting at the depot $\rho_f \in V$, with a total distance not exceeding a total distance budget $B_f$.
\end{definition}
An illustration of the \OIG is presented in Figure \ref{fig:bayg29}, on the \texttt{bayg29} instance of TSPLIB95 (\url{http://comopt.ifi.uni-heidelberg.de/software/TSPLIB95/}). The instance involves 29 cities in Bavaria and connections between all city pairs. We arbitrarily determine city 1 as the depot node and consider unit prizes, unit interdiction costs, and three levels of the interdiction budget, i.e., $Q_\ell\in \{0,5,8\}$. 
Note that with an interdiction budget of zero the obtained problem is just the OP without interdiction. As budget we used $B_f = 0.5\nu$ where $\nu$ denotes the optimal
TSP tour length which is provided with the instances.
In the figure, the optimal leader solution consists of the green nodes, and the resulting follower tour is shown in each figure. The total prizes collected by the follower are 16, 12, and 11, respectively. 

\begin{figure}[htb]
    \centering
    \begin{subfigure}{0.3\textwidth}
    \includegraphics[width=\textwidth]{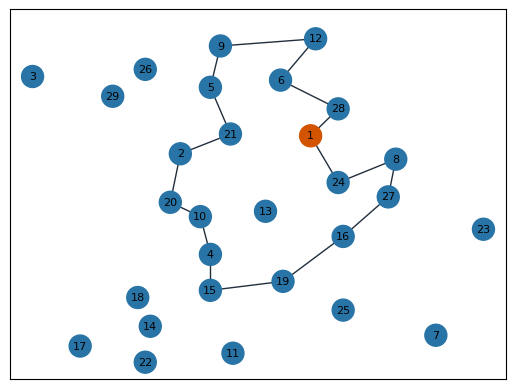}
    \caption{$Q_\ell=0$ (i.e., the OP without interdiction)}
     \end{subfigure}
        \begin{subfigure}{0.3\textwidth}
    \includegraphics[width=\textwidth]{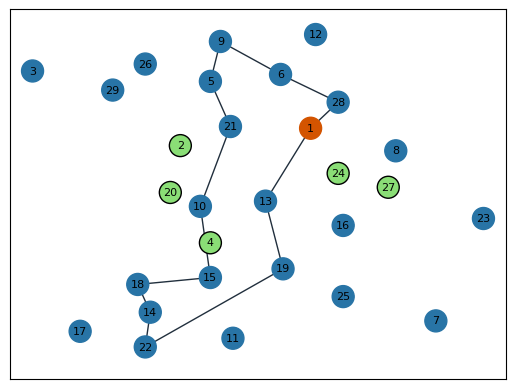}
      \caption{$Q_\ell=5$ \newline}
         \end{subfigure}
    \begin{subfigure}{0.3\textwidth}
    \includegraphics[width=\textwidth]{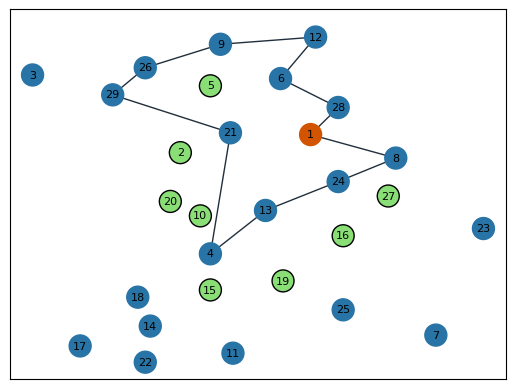}
          \caption{$Q_\ell=8$ \newline}
             \end{subfigure}
    \caption{The \texttt{bayg29} instance from TSPLIB95 library, with a given depot node (orange) and  considering unit prizes and unit interdiction costs. The interdicted nodes (green) and the resulting follower tour are shown under different leader budget $Q_\ell$.}
    \label{fig:bayg29}
\end{figure}

\OIG can model several applications, such as
preventing competitors from effective canvassing, also known as door knocking, which is considered crucial for political campaigns \citep{lupfer1972merits,bhatti2019door,nyman2017door}. In such a setting, the leader who wants to minimize the support to the competitor via canvassing and who has limited resources should identify and convince key groups of individuals. 
It can also be used to identify critical locations for patrolling. While studies including \citep{keskin2012analysis,cruz2021mathematical} focus on the maximization of patrol coverage, in the interdiction version the leader's problem could model an adversary who wants to damage the benefit/coverage of patrolling. Solving this problem reveals important locations whose interdiction undermines the security operations in the worst possible way. 
Lastly, the leader can model security forces who want to prevent criminal activities of a moving follower. An example is to find the best spots in a touristic district to continuously monitor so that the pickpockets who stroll around are restrained.
To the best of authors' knowledge, there is limited work considering interdiction within a routing problem. In Section \ref{sec:lit} we provide a literature review on existing problems and studies in this area.

\subsection{Contribution and outline}

The \OIG is a two-player Stackelberg game \citep{von1952theory} and thus it can be modeled as a bilevel optimization problem (BOP). 
We first formulate the \OIG as a BOP and then propose a single-level reformulation based on so-called \emph{interdiction cuts} to tackle this challenging problem. This technique was introduced in \cite{fischetti2019interdiction} for interdiction games fulfilling a certain \emph{monotonicity} assumption and we show that cuts of this form can also be used for the \OIG. Based on this reformulation, we develop a branch-and-cut algorithm to solve the \OIG and introduce various enhancements for the algorithm. For solving the lower level problem within this algorithm, we make use of the branch-and-cut ideas proposed by \citet{fischetti1998solving} for the \OP. The main contributions of our work can be summarized as follows:

\begin{itemize}
    \item We introduce the \OIG, which is a competitive version of the well known \OP and can model various applications ranging from security games to campaign planning.
    \item We formulate the \OIG as a zero-sum BOP with binary decision variables in both levels. Then, we propose a single-level reformulation of the problem using \emph{interdiction cuts}. To the best of our knowledge, this method has not been used for any other routing problem yet.
    \item Based on this formulation, we propose an algorithmic framework to solve the \OIG exactly via branch-and-cut. This framework includes several components such as cut pools and integrated heuristic procedures which could reduce the computational burden. 
    \item We develop a genetic algorithm to heuristically solve the \OIG.
    \item We provide a computational study on \OP instances from literature adapted to the \OIG to assess the efficacy of our solution algorithms and their ingredients.
\end{itemize}

The paper is organized as follows. In Section \ref{sec:model} we introduce the BOP formulation of the \OIG and propose a single-level reformulation. We then describe a branch-and-cut method together with several enhancement strategies. In Section \ref{sec:GA}, we present our genetic algorithm. We evaluate the performance of our solution approaches in Section \ref{sec:results}. Finally, we draw conclusions and provide possible future research directions in Section \ref{sec:conclu}.

\subsection{Literature review \label {sec:lit}}

In this section, we provide an overview of the literature on interdiction games involving routing decisions and on routing problems in a bilevel optimization setting. 
First, we focus our attention on studies that consider routing interdiction. In one of them, \citet{lozano2017solving} address a fortification-interdiction variant of the traveling salesperson problem where the arcs are subject to protection and interdiction. Unprotected arcs can be interdicted by the attacker, which increases the cost of those arcs. They propose an exact iterative algorithm based on sampling of feasible tours. In \citep{kheirkhah2016improved} the vehicle routing problem with complete arc interdictions is considered, and a Benders decomposition algorithm is proposed which is capable to solve small size problems. In \citep{kheirkhah2016bi}, a hazmat routing interdiction problem is presented in which the leader aims to minimize the risk and the follower minimizes the routing cost. They propose metaheuristic methods to solve it. Several variants of arc interdiction vehicle routing problem are addressed by \cite{bidgoli2018arc,sadati2020r,sadati2020trilevel} and \cite{nadizadeh2021bi}, and are handled via heuristic/metaheuristic methods. To the best of our knowledge, neither node interdictions nor a follower solving an orienteering problem were considered before.

Next, we review works addressing a bilevel optimization problem (BOP) involving routing in its lower level, i.e., a more general setting. 
In BOPs each player has his/her own objectives and constraints and the leader, who acts first, anticipates the optimal follower response (see, e.g., \citet{dempe2020bilevel, kleinert2021survey} for an overview on BOPs). 
An \IG is a special class of BOPs where the leader and the follower optimize the same objective function in opposite directions, while the leader affects the follower problem via interdictions of his/her assets. 
Regarding exact approaches for BOPs with a routing component, \citet{cerulli2023bilevel} introduce the bilevel profitable tour problem where the leader is the logistics platform assigning orders to carriers and the followers are the carriers solving a profitable tour problem. They develop a branch-and-cut algorithm to solve the problem exactly.
Aside from this work which uses an exact algorithm, there are also several works on BOPs with a routing component which propose metaheuristic algorithms to solve the adressed problem:
The paper \citep{nikolakopoulos2015metaheuristic} addresses vehicle routing problem with backhauls and time windows in a military context. The problem is formulated as a BOP where the goal of the leader is to minimize the number of vehicles, and the follower wants to minimize the total routing cost. 
\citet{ning2017multilevel} consider a vehicle routing problem with uncertain travel times. In the proposed bilevel model, the leader minimizes the expected total waiting time, and each follower (vehicle) minimizes its own waiting time.  
In the production-distribution planning problem considered by \citet{calvete2011bilevel}, the distribution and manufacturing companies act respectively as the leader and the follower of a bilevel model. Each player seeks to minimize their costs.
\citet{camacho2022tabu} study a similar problem in which the distributor has $CO_2$ emission goals in addition to profit maximization. This leads to a bi-objective leader problem. 
\citet{parvasi2019bi} consider the problem of bus stop location and school bus routing, which is modeled as a BOP. 

Aside from these works, which model competitive settings with multiple agents, sometimes routing problems just involving a single level (i.e., a single decision maker) are also modeled as BOPs:
\citet{marinakis2007new} formulate the vehicle routing problem as a BOP such that in the first level customers are assigned to vehicles and in the second level optimal routes of these assignments are determined. 
Similarly, \citet{marinakis2008bilevel} model the location routing problem with a BOP formulation whose leader makes the strategic decisions  (facility locations) and whose follower makes operational decisions (routes). Both problems are solved via genetic algorithms. 
\citet{jia2021bilevel} formulates the capacitated electric vehicle routing problem as a BOP whose upper level involves the routing decisions and the lower level involves determining the charging schedule. They propose an ant colony optimization algorithm. 

Note that there are generic methods for solving interdiction games under certain assumptions, such as the iterative bounding algorithms in \citet{tang2016class, lozano2017backward, taninmics2022improved}, or the branch-and-cut method in \citet{fischetti2019interdiction}. Similarly, there exist methods for solving integer bilevel linear programming problems \citep{wang2017watermelon, fischetti2017new,tahernejad2020branch}. However, due to the inherent difficulty of interdiction games and integer bilevel linear programming problems, these generic methods usually are only capable of solving rather small-sized (generic) instances. 
For this reason, we design an exact algorithm tailored to the \OIG (based on ideas from \citet{fischetti2019interdiction}) and also propose a genetic algorithm to heuristically obtain high-quality solutions within a shorter running time compared to our exact approach. 

\section{An interdiction-cut-based exact solution method for the \OIG}
\label{sec:model}

In order to formulate the \OIG as a bilevel integer program, we introduce the following notation. Let $\mathbf{z}\in\{0,1\}^{|V|}$ be a vector of decision variables,
so that $z_i = 1$ if the leader interdicts node $i\in V$,
and $z_i = 0$ otherwise;
hence, $z_{i}=1$ implies that the leader prevents the follower from getting the prize of node $i$.
Likewise, let $\mathbf{y}\in\{0,1\}^{|V|}$ be a vector of decision variables so that
$y_i = 1$ if the follower visits node $i\in V$ in his/her tour, and $y_i = 0$ otherwise;
additionally, let $\mathbf{x}\in\{0,1\}^{|E|}$ be a vector of decision variables so that
$x_{e} = 1$ if the follower traverses the edge $e$ in the tour, and $x_{e} = 0$ otherwise. 
Let $\tau(\mathbf{x},\mathbf{y})$ denote the tour induced by a given pair $(\mathbf{x},\mathbf{y})$;
and let $\mathcal{T}$ be the set of 
all feasible tours,
i.e., cycles that do not contain subtours and that pass through the depot $\rho_f$, 
where each node in the tour is visited at most once. 
Considering this notation, the \OIG can be formulated by the following (bilevel) MIP model:
\begin{flalign}
\mbox{\mytag{\OIGB}} \qquad\qquad  {\phi}^{\ast} = \min_{\mathbf{z}\in \{0,1\}^{|V|}} \;\; \max_{(\mathbf{x},\mathbf{y})\in \{0,1\}^{|E|\times |V|}} \;\; & \sum_{i \in V} p_i(1-z_i)y_i \qquad\qquad\qquad\qquad \label{eq:iop1}\tag{OBJ}\\
    \mbox{s.t.}\qquad 
    & \sum_{i\in V} z_{i} \leq Q_\ell \label{eq:iop2}\tag{IBUDGET} \\
    & \sum_{e\in E} d_{e} x_{e} \leq B_f \label{eq:iop3}\tag{DBUDGET}\\
    &\tau(\mathbf{x},\mathbf{y}) \in \mathcal{T}.  \label{eq:iop4}\tag{TOUR}
\end{flalign}
The objective function~\eqref{eq:iop1} encodes the (bilevel) $\min\max$ optimization 
goal, which is the minimization of maximum prize collected from the nodes that are visited and non-interdicted.
Constraint~\eqref{eq:iop2} imposes that the total
number of interdictions does not exceed the interdiction budget $Q_\ell$.
Likewise, the constraint~\eqref{eq:iop3} imposes that the total distance
of the follower's tour does not exceed the follower's distance budget $B_f$.
Finally, the constraint~\eqref{eq:iop4} ensures that the
vectors $\mathbf{x}$ and $\mathbf{y}$ must induce a feasible tour that includes the
depot $\rho_f$.

\subsection{A single-level reformulation of the \OIG}

Before we present a single-level reformulation of the \OIG, we first formulate constraint~\eqref{eq:iop4}
using subtour elimination constraints.
For a given subset of nodes $S\subseteq V$, let $\delta(S) = \left\{\{i,j\}:e\in E\mid
i\in S,\; j\in V\setminus S\right\}$;
i.e., $\delta(S)$ corresponds to the set of edges that are incident to the nodes contained in set $S$.
Using this definition, $\tau(\mathbf{x},\mathbf{y}) \in \mathcal{T}$
can be encoded by the following set of constraints:
\begin{align}
    &\sum_{e\in \delta(j)}x_e=2y_j,\; \forall j\in V \label{eq:degree} \tag{SEC.1}\\
    &\sum_{e\in \delta(S)}x_{e}\geq 2 y_j,\; \forall j\in V\setminus S,\; \forall S\subset V\mid_{\rho_f \in S} \label{eq:GSEC} \tag{SEC.2}\\
    &y_{\rho_f}=1. \label{eq:depot} \tag{SEC.3} 
\end{align}
Constraints~\eqref{eq:degree} model the fact that if a node $j'\in V$ 
is included in the follower's tour (i.e., $y_{j'}=1$), 
then exactly two of its adjacent edges must also be included in the tour 
($\sum_{e\in \delta(j')}x_e = 2$).
Constraints~\eqref{eq:GSEC} are the so-called \emph{generalized subtour elimination constraints} (GSECs) and they ensure that the tour defined
by $\mathbf{x}$ and $\mathbf{y}$ does not contain subtours.
Constraint~\eqref{eq:depot} ensures that the depot $\rho_f$
is included in the follower's tour.

Let $\Phi(\mathbf{z})$ be the value function of the lower level problem for a given interdiction decision $\mathbf{z}$, i.e., 
\begin{align*}
    \Phi(\mathbf{z}) = \max_{(\mathbf{x},\mathbf{y})\in \{0,1\}^{|E|\times |V|}} \;\; & \sum_{i \in V} p_i(1-z_i)y_i  \\
    \mbox{s.t.}
    &\mbox{ \eqref{eq:iop3}, \eqref{eq:degree}, \eqref{eq:GSEC}, and \eqref{eq:depot}}.
\end{align*}
Then, we can write the value function reformulation of the \OIG as follows:
\begin{align*}
    {\phi}^{\ast} = \min_{\mathbf{z}\in \{0,1\}^{|V|}}\;\; & t \\
    \mbox{s.t.}\qquad &t \geq \Phi(\mathbf{z})\\
    & \eqref{eq:iop2}.
\end{align*}
Note that the value function reformulation is non-convex, even when the binary restrictions are relaxed, due to the value function constraint $t\geq \Phi(\mathbf{z})$ which ensures that the objective function value of the upper level is at least as large as the optimal objective function value of the lower level for the selected interdiction decision. However, it can be further reformulated by considering the feasible follower solutions as follows. 
Let $\mathbf{Y}$ be the set of \textit{all} vectors $\mathbf{\hat y}$
such that there exists a tour $\tau(\mathbf{\hat x},\mathbf{\hat y}) \in \mathcal{T}$ for some $\mathbf{\hat x}\in \{0,1\}^{|E|}$,
satisfying the follower's distance budget $B_f$. 
Using this notation, we can obtain the following single-level reformulation of the \OIG.
\begin{align}
   \mbox{\mytag{OIGS}} \qquad\qquad  {\phi}^{\ast} = \min_{\mathbf{z}\in \{0,1\}^{|V|}}\; &t \notag\\
    \mbox{s.t.}\qquad & t \geq \sum_{i \in V} p_i(1-z_i)\hat y_i,\; \forall \mathbf{\hat y} \in \mathbf{Y} \label{eq:interdictionCut} \tag{ICUT}\\
    & \eqref{eq:iop2} \notag
\end{align}
In this formulation,~\eqref{eq:interdictionCut} corresponds to the set of \emph{interdiction cuts} (following the terminology of \citet{fischetti2019interdiction}). These cuts model the value function constraint $t\geq \Phi(\mathbf{z})$.

\begin{proposition}
The formulation \myref{OIGS} models the \OIG.
\end{proposition}

\begin{proof}
We have to show that for any interdiction decision $\mathbf{z}$ the set of constraints~\eqref{eq:interdictionCut} ensures that $t$ will have the value of $\Phi(\mathbf{z})$ (taking into account that the objective function forces $t=\max_{\mathbf{\hat y} \in \mathbf{Y}} \sum_{i \in V} p_i(1- z_i)\hat y_i$). Suppose this is not the case. This means there exist an interdiction decision $\mathbf{\bar  z}$ for which either $\Phi(\mathbf{\bar  z}) > \sum_{i \in V} p_i(1-\bar z_i)\hat y_i,\; \forall \mathbf{\hat y} \in \mathbf{Y}$ or $\exists \mathbf{\hat y} \in \mathbf{Y}: \sum_{i \in V} p_i(1-\bar  z_i)\hat y_i> \Phi(\mathbf{\bar  z})$. We proceed by case distinction.
\begin{itemize}
\item Assume $\Phi(\mathbf{\bar  z}) > \sum_{i \in V} p_i(1-\bar z_i)\hat y_i,\; \forall \mathbf{\hat y} \in \mathbf{Y}$: Let $\mathbf y^*(\mathbf{\bar  z})$ be an optimal solution of $\Phi(\mathbf{\bar  z})$. Clearly $\mathbf y^*(\mathbf{\bar  z}) \in \mathbf Y$ and $\sum_{i \in V} p_i(1-\bar  z_i) y^*(\mathbf{\bar z})_i=\Phi(\mathbf{\bar z})$. Thus we arrive at a contradiction to our assumption.
\item Assume $\exists \mathbf{\hat y} \in \mathbf{Y}: \sum_{i \in V} p_i(1-\bar  z_i)\hat y_i> \Phi(\mathbf{\bar  z})$: Since the interdiction decisions to not affect the feasibility of the lower level problem, we have that $\mathbf{\hat y}$ is feasible for the lower level problem given interdiction decision $\mathbf{\bar  z}$. Consequently the value of $\Phi(\mathbf{\bar  z})$ must be at least $\sum_{i \in V} p_i(1-\bar  z_i)\hat y_i$. Thus we we arrive at a contradiction to our assumption.
\end{itemize}
We have arrived at a contradiction for both cases, which concludes our proof.
\end{proof}

\subsection{A branch-and-cut algorithm}
\label{sec:algo}

In order to solve the \OIG, we propose a branch-and-cut (B\&C) algorithm based on \myref{OIGS}, where we drop \eqref{eq:interdictionCut} and add those
constraints \emph{on-the-fly}. 
In the remainder, we denote the linear programming (LP) relaxation of any B\&C subproblem by $\overline{\mbox{\OIGS}}$, 
which includes a subset of \eqref{eq:interdictionCut}, in addition to \eqref{eq:iop2}, and branching decisions made to reach the current B\&C node. We note that in order to separate \eqref{eq:interdictionCut} we employ another B\&C algorithm within our main B\&C algorithm, details are given below.

\subsubsection{Separation of interdiction cuts}

In this section, we describe how we separate the inequalities \eqref{eq:interdictionCut} while implementing the B\&C algorithm to solve the \OIG.
Let $(\bar{t},\bar{\mathbf{z}})$ be a feasible 
solution to $\overline{\mbox{\OIGS}}$. 
We need to solve the following separation problem, which is identical to the follower's problem for $\mathbf{z}=\bar{\mathbf{z}}$:
\begin{align*}
  \mbox{\mytag{SEP}} \qquad\qquad  \Phi(\bar{\mathbf{z}}) = \max_{(\mathbf{x},\mathbf{y})\in \{0,1\}^{|E|\times |V|}} & \; \sum_{i \in V} p_i(1-\bar{z}_i)y_i \\
    \mbox{s.t.}\qquad 
        & \mbox{\eqref{eq:iop3},~\eqref{eq:degree},~\eqref{eq:GSEC}~and~\eqref{eq:depot}}.\notag  
\end{align*}
If $\bar t < \Phi(\bar{\mathbf{z}})$, then we need to add a violated interdiction cut to separate the point $(\bar{t},\bar{\mathbf{z}})$. Otherwise, $\bar{t}$ captures the optimal follower objective value correctly, and we treat the current point as a feasible solution to our problem. 

We enhance the formulation \myref{SEP} by making use of two classes of valid inequalities for the orienteering problem \citep{fischetti1998solving}. 
The first class of inequalities, that we refer to as \emph{logical constraints}, 
is given by\begin{align}
    x_e\leq y_j, \qquad \forall e \in \delta(j),\;\forall j \in V. \label{eq:logical} \tag{Logical}
\end{align}
Basically, these constraints ensure that if a given node $j$ is not visited by
the follower's tour ($y_j = 0$), then none of the incident edges is part of the tour ($x_e = 0,\;\forall e \in \delta(j)$).
The second class of inequalities corresponds to the \emph{cycle cover inequalities},
which are given by\begin{align}
    \sum_{e\in E_\tau}x_{e} \leq \sum_{j\in V_\tau} y_j -1, \label{eq:CC} \tag{CC}
\end{align}
for a given tour $\tau$ that is encompassed by the nodes in $V_\tau$ and by the edges in $E_\tau$, 
and with a total distance $\sum_{e\in E_\tau}d_e$ larger than $B_f$.
These constraints ensure that if the follower's tour includes the nodes in $V_\tau$,
then at least one of the edges in $E_\tau$ must be excluded from the tour
in order to not violate the corresponding follower's distance budget constraint~\eqref{eq:iop3}.

\subsubsection{Separation of follower cuts}

To solve the separation problem (SEP), we carry out another B\&C procedure where we drop the subtour elimination constraints \eqref{eq:GSEC} from the initial formulation and add violated ones once they are detected, on-the-fly. In addition, we make use of the valid inequalities described in the previous section. Suppose that we are given the (possibly infeasible) follower solution $(\mathbf{\bar x}, \mathbf{\bar y})$. In the following, we describe how we separate the inequalities of each type at $(\mathbf{\bar x}, \mathbf{\bar y})$.

\paragraph{Generalized subtour elimination constraints.}
These inequalities are obtained using the maximum flow-based approach proposed by \cite{fischetti1998solving}. In this approach, given $(\mathbf{\bar x}, \mathbf{\bar y})$, first the so-called support graph is first obtained which contains the nodes and edges of the original graph such that $\bar x_e>0$ and $\bar y_j>0$, respectively. On this graph, a minimum capacity cut separating the depot and any non-depot node leads to a (possibly violated) GSEC. The nodes are considered one by one, according to the non-decreasing order of $\bar y_j$ and then the maximum flow computations are done. We refer the reader to \cite{fischetti1998solving} for further details.

\paragraph{Logical constraints.}
These inequalities are obtained by complete enumeration of node-edge pairs such that $\bar x_e >y_j $ where $e\in \delta(j)$, which can be done in $O(|E|)$ time.

\paragraph{Cycle cover inequalities.}
These inequalities are obtained via the heuristic procedure in \cite{fischetti1998solving}. This heuristic takes $\mathbf{\bar x}$ as input and computes a maximum-weight spanning tree. Then for each edge $e$ that is not in the tree, if adding $e$ creates a cycle that passes through $\rho_f$, then the violation of the cut for the resulting tour is checked.

At every (follower) B\&C node with a fractional solution $(\bar{\mathbf{x}},\bar{\mathbf{y}})$, we first look for all possible violated valid inequalities \eqref{eq:logical}. If none is found, then we try to find violated inequalities \eqref{eq:GSEC}. If violated cuts are not identified, then violated \eqref{eq:CC} are tried to be obtained. In case there are no violated inequalities of these three types, no action is needed.
For integer feasible $(\bar{\mathbf{x}},\bar{\mathbf{y}})$, we only look for violated inequalities of type \eqref{eq:GSEC}. 
If there is none, $(\bar{\mathbf{x}},\bar{\mathbf{y}})$ is a feasible solution to \myref{SEP} and it replaces the incumbent solution if it is a better one.

\subsection{Enhancement strategies}
\label{sec:enhancement}

While solving the \OIG with the B\&C approach that we propose, we make use of some strategies that could help to speed up the algorithm, in particular the separation procedure. Below we describe the algorithm components developed to this end. In our computational experiments, we try different combinations of these enhancement strategies. The details of these combinations including the values used for the parameters which occur in some of the strategies described in the following are given in Section \ref{sec:numericresults}.

\paragraph{Separation problem objective lower bound (lower cutoff).}
For a given leader solution $(\bar{t},\bar{\mathbf{z}})$ that is the optimal solution of the current B\&C subproblem, the optimal objective value $\Phi(\bar{\mathbf{z}})$ of the separation problem (SEP) cannot be less than $\bar{t}$ since the interdiction cuts underestimate the follower objective value. Therefore, we can set a lower cutoff value $\bar{t}$ on the objective function value while solving (SEP), for more efficient pruning of nodes in the B\&C tree of (SEP).

\paragraph{Cut pool.} Every time (SEP) is solved, inequalities of type \eqref{eq:CC} and \eqref{eq:GSEC} are generated. We keep a pool of previously obtained inequalities of these types to be used in later attempts to solve the separation problem. Based on preliminary experiments, instead of including all the inequalities of the cut pool in the initial formulation of (SEP), we iterate over the pool and add the violated ones at the root node of the B\&C tree, after each time the node subproblem is solved. 

\paragraph{Solution pool.} Every time (SEP) is solved for a given leader solution $\bar{\mathbf{z}}$,
which yields a feasible solution $(\bar{\mathbf{x}},\bar{\mathbf{y}})$,
we retrieve the obtained tour $c = \tau(\bar{\mathbf{x}},\bar{\mathbf{y}})$
and add it to a (solution) pool denoted by $\mathcal{C}$. These solutions are utilized to generate new feasible follower solutions that possibly yield violated interdiction cuts, as described in the following paragraph. 
We denote by $V_c=\{j\in V: \bar y_j=1\}$ and $E_c=\{e\in E: \bar x_e=1\}$ the sets of nodes and edges, respectively, associated to tour $c$.

\paragraph{Heuristic follower solutions.}
For the correctness of the overall B\&C algorithm, it is necessary to have an exact separation method for integer leader solutions, i.e., a method that returns a violated interdiction cut when in the current solution $(\bar{t},\bar{\mathbf{z}})$, $\bar{t} $ strictly underestimates the follower objective value for $\bar{\mathbf{z}}$. On the other hand, it may be possible to separate fractional solutions as well which could improve the dual bound faster. Since it is usually costly to carry out an exact method for every fractional solution encountered during B\&C, the common approach is to use a heuristic separation algorithm that can yield violated cuts \citep{fischetti2019interdiction}. To this end, we propose a heuristic separation scheme that can be used for fractional and integer solutions, where we iterate over the solution pool $\mathcal{C}$. As described in Algorithm \ref{alg:FH}, for each tour $c$ in the pool, we first repair it by removing the nodes whose prizes cannot be collected under the current interdiction strategy $\bar{\mathbf{z}}$, unless this makes the path longer. 
Then, we apply \texttt{2-Opt} and \texttt{Insert} operations until the objective value cannot be improved, or the prize threshold is exceeded. The latter means that we are able to find a follower solution yielding a violated cut.

\begin{algorithm}[htbp]
\caption{\texttt{FindHeuristicFolSoln}}\label{alg:FH}
\SetKwInOut{Input}{Input}\SetKwInOut{Output}{Output}
\Input{Solution $\bar{\mathbf{z}}$ to ($\overline{\mbox{OIP}}$) (associating objective function  $\bar{t}$), a solution pool $\mathcal{C}$}
\Output{Possibly a feasible follower solution $(\bar{\mathbf{x}},\bar{\mathbf{y}})$ with a total prize larger than $\bar{t}$}
\For{each $c$ in $\mathcal{C}$}
{
Define a new set $c^\prime \leftarrow c$\;
\For{each node $i\in V_{c^\prime}$} 
    {
        \If{$\bar{z}_i=1$ and removing $i$ from $V_{c^\prime}$ does not increase the length of tour $c^\prime$}
        {
            Remove $i$ from $V_{c^\prime}$\;
        }
    }
set $\pi^\prime = \sum_{i\in V_{c^\prime}}p_i$, and set $\pi^{\prime\prime} \leftarrow 0$\;
\While{$\pi^{\prime\prime}<\pi^\prime \leq \bar{t}$}
    {
        $\pi^{\prime\prime} \leftarrow \pi^{\prime}$\;
        $c^{\prime} \leftarrow$ \texttt{2OPT($c^\prime$)}\;
        $c^{\prime}, \pi^{\prime} \leftarrow$ \texttt{INSERT($c^{\prime}$)}\;
    }
\If {$\pi^\prime > \bar{t}$}
    {
        Return the follower's solution $(\bar{\mathbf{x}},\bar{\mathbf{y}})$ associated with $c^\prime$\;
    }
}
\end{algorithm}

\paragraph{Follower preprocessing.} 
Given a leader solution $\bar{\mathbf{z}}$, let $i^\prime$ be a node selected by the leader,
i.e., $\bar{z}_{i^\prime} = 1$.
If there exists no pair $a,b\in V$ such that $d_{ai^\prime}+d_{i^\prime b}< d_{ab}$, 
it is safe to assume that the follower would not visit $i^\prime$ in an optimal tour 
since its prize is not available to the follower anymore.
In this case, we can fix $y_{i^\prime} = 0$ in \myref{SEP}. 
Otherwise, one of the following conditions could hold in an optimal follower tour: 
\textbf{(i)} $i^\prime$ is not visited, \textbf{(ii)} $i^\prime$ is visited between a node pair $a$ and $b$ where $d_{a i^\prime}+d_{i^\prime b}< d_{ab}$, 
or \textbf{(iii)} $i^\prime$ is visited between a node pair $a$ and $b$ where $d_{a i^\prime}+d_{i^\prime b}\geq d_{ab}$, i.e., visited although it does not make the tour the shorter.
In the last case, removing $i^\prime$ would not compromise the optimality of the follower's tour 
since prize $p_{i^\prime}$ is not collected anyway. 
So, exactly one of the following conditions must hold: $(y_{i^\prime}=0) $, $ (x_{a_1 i^\prime}+x_{i^\prime b_1}=2), \ldots ,(x_{a_k i^\prime}+x_{i^\prime b_k}=2)$, where $(a_k,b_k)$ are all pairs with $d_{a_k i^\prime}+d_{i^\prime b_k}< d_{a_k b_k}$. This result can be used to strengthen the formulation (SEP) with additional constraints. The details are provided in Algorithm \ref{alg:FollowerPreprocessing}.

\begin{algorithm}[htbp]
\caption{\texttt{FollowerPreprocessing}}\label{alg:FollowerPreprocessing}
\SetKwInOut{Input}{Input}\SetKwInOut{Output}{Output}
\Input{Leader solution $\bar{\mathbf{z}}$} 
\Output{A set of additional constraints $\mathcal{CS}$ to be added to (SEP)}
\For{each $i\in V\mid \bar{z}_i=1$}{
\eIf{$d_{ai}+d_{ib}\geq d_{ab}$ for all $a,b \in V \setminus \{i\}$}{
set $\mathcal{CS} \leftarrow \mathcal{CS} \cup \{x_i\leq 0\}$\;
}
{
compute $\mathcal{W}=\{(a,b)\mid a,b \in V\setminus \{i\}, d_{a i}+d_{i b}< d_{a b}\}$\;
create binary variable $b_{i0}$, set $\mathcal{CS} \leftarrow \mathcal{CS} \cup \{y_i\leq 1- b_{i0}\}$\;
initialize $k=0$\;
\For{each $(a,b)\in \mathcal{W}$}{
$k\leftarrow k+1$\;
create binary variable $b_{ik}$, set $\mathcal{CS} \leftarrow \mathcal{CS} \cup \{x_{ai}+x_{ib}\geq 2 b_{ik})\}$\;
}
set $\mathcal{CS} \leftarrow \mathcal{CS} \cup \{\sum_{k^\prime=0}^k b_{ik^\prime}=1 \}$\;
}
}
\end{algorithm}

\section{A genetic algorithm for the \OIG}
\label{sec:GA}

In this section, we propose a genetic algorithm to solve the \OIG. In a genetic algorithm, it is often helpful if the fitness value of an individual $z^\prime$ of the population would be equal to the objective function value. For the \OIG this would mean that the fitness value should be $\Phi(z^\prime)$, which is the optimal objective value of the OP under the interdiction decision $z^\prime$. Hence, to evaluate the fitness of an individual, we would need to solve an NP-hard problem.
Thus, in the design of our algorithm, for better time efficiency, we opt for a faster and heuristic method to calculate the fitness values. Procedure \texttt{EstimateObjective} is very similar to Algorithm \ref{alg:FH} and uses a follower solution pool to generate a good follower solution for the current interdiction strategy. Unlike Algorithm \ref{alg:FH} it does not take a target objective as input, but instead iterates over all pool solutions and returns the best new solution obtained and its total prize. Note that the resulting pair of leader and follower solutions may not be bilevel feasible as the follower tour that \texttt{EstimateObjective} outputs is only a heuristic solution for the given interdiction strategy (i.e., leader solution), and to be feasible, it needs to be an optimal solution for the given leader solution. However, this is not a problem since we use the fitness value only to lead our search towards a better leader solution.
To ensure bilevel feasibility of the final solution, a post-processing step is carried out to obtain the optimal follower tour.

The pseudocode of the genetic algorithm is provided as Algorithm \ref{alg:GA}. It starts with the generation of an initial set of follower solutions $\mathcal{C}$ to be used within \texttt{EstimateObjective} in later steps. In this step, we use the heuristic algorithm in \citep{fischetti1998solving} to solve the OP, which takes the optimal solution of the LP relaxation of the OP at the current B\&C node, generates a feasible tour in the first stage, and improves it in the second stage. We refer the interested reader to \citep{fischetti1998solving} for further details. In our case, we implement this method $k_0$ times, i.e., we solve the LP relaxation of the follower problem for $k_0$ randomly generated feasible $\mathbf{z}$ values. For each of them, once we get a feasible tour and improve it as described by \cite{fischetti1998solving}, we remove, in turn, one of the tour nodes and reapply the improvement procedure. We add the resulting tour to $\mathcal{C}$ if it is not obtained in the previous iterations.

\begin{algorithm}[htbp]
\caption{\texttt{GeneticAlgorithm}}\label{alg:GA}
\SetKwInOut{Input}{Input}\SetKwInOut{Output}{Output}
\Input{An instance of the \OIG 
} 
\Output{A feasible solution to \OIG{} and an upper bound on ${\phi}^{\ast}$}

Initialize the follower solution pool $\mathcal{C}$ with a set of feasible tours\;
initialize the population $P=\emptyset$\;
\For {$m=1,\ldots,p_0$}
{
$\hat{z}\leftarrow \texttt{Greedy}()$\;
$P\leftarrow P\cup \hat{z}$\;
}
\For {$it=1,\ldots, n_{maxIter}$}
{
\If{$it \equiv 0 \ (\mathrm{mod}\ 10)$}
{
re-evaluate the fitness values of the $k$ fittest individuals by \texttt{EstimateObjective}\;
}
$parent_1\leftarrow \texttt{K-Tournament}(P)$, $parent_2\leftarrow \texttt{K-Tournament}(P)$\;
$z^\prime \leftarrow \texttt{Crossover}(parent_1,parent_2)$\;
$z^\prime \leftarrow \texttt{Mutation}(z^\prime)$\;
$z^\prime \leftarrow \texttt{Repair}(z^\prime)$\;
evaluate the fitness of $z^\prime$ using \texttt{EstimateObjective}\;
\eIf{$|P|=p_{max}$}
{$P\leftarrow (P \setminus \arg \max_{\hat{z}\in P} \texttt{EstimateObjective}(\hat{z},\mathcal{C}))\cup z^\prime$ (replace the individual with the worst fitness value with $z^\prime$)\;}{$P\leftarrow P \cup z^\prime$ } 
}
re-evaluate the fitness values of all individuals\;
choose the fittest individual $z^\ast = \arg \min_{\hat{z}\in P} \texttt{EstimateObjective}(\hat{z},\mathcal{C})$\;
solve the follower problem to optimality to obtain $\Phi(z^\ast)$\;
return $z^\ast$ and $\Phi(z^\ast)$\
\end{algorithm}

After initializing $\mathcal{C}$, we initialize the population with $p_0$ interdiction strategies obtained with a randomized greedy algorithm \texttt{Greedy}. It starts with the computation of initial marginal gains, i.e., an estimate of the decrease in the follower objective, due to interdicting each node in $V$ and storing them in a sorted list in decreasing order of gains. Then we iterate in a lazy fashion and pick the first node in the list. If its gain is not updated after the last interdiction decision, we recompute it with probability $1-p_s$ and re-sort the list. With the skipping probability $p_s$ we remove the node from the list. If the first node in the list has a newly computed gain value, we choose it to be the next node to interdict. We iterate until the interdiction budget $Q_\ell$ is reached or the list is empty. For the computation of marginal gains at any stage of \texttt{Greedy} we use \texttt{EstimateObjective}. Due to the skipping probability, at each call to \texttt{Greedy} we obtain a different interdiction strategy, i.e., an individual.

The selection of the parents is made according to a $K$-way tournament selection. For each of the two parents, $K$ individuals are randomly chosen and the fittest one is selected as a parent. A one-point crossover operator \texttt{Crossover}() is applied to the parents to generate a single offspring. Then, the mutation operator \texttt{Mutation}() applies zero, one, or two random bit flips with equal probabilities. If the offspring does not represent a feasible leader solution, we repair it by switching ones to zeros at the bits with smallest prize until the budget constraint is satisfied, which is denoted by \texttt{Repair}(). The fitness of the offspring is then computed via \texttt{EstimateObjective}(). We use a \emph{steady-state (incremental)} population structure, where an offspring replaces the individual with the worst fitness value immediately if the maximum population size $p_{max}$ is reached. 

The accuracy of \texttt{EstimateObjective} depends on the quality and diversity of the solutions in $\mathcal{C}$. Therefore, every time we call this procedure, we add the resulting tour to $\mathcal{C}$ considering an upper bound on the size of the pool. Since it is dynamically updated, the fitness value of an individual may change and get closer to the true value when it is re-calculated after some iterations. To better estimate the true objectives, we re-calculate the fitness values of the best $k$ individuals at every 10 iterations. Similarly, once the iterations are over, we re-calculate the fitness of each individual to determine the fittest one. Finally, to have a bilevel feasible solution, we solve the follower problem optimally for the selected leader solution corresponding to the fittest individual.

\section{Computational results and discussion}
\label{sec:results}

All the algorithms we propose are implemented in C++. Whenever we need to solve some MIPs or LPs they are solved via IBM ILOG CPLEX 12.10 at the default settings. We make use of CPLEX callbacks to implement our B\&C algorithm (including the B\&C algorithm to solve \myref{SEP}). During our experiments, we used the single core of an Intel Xeon E5-2670v2 machine with 2.5 GHz processor and 3GB of RAM. We set a time limit of one hour for all of our experiments.

\subsection{Description of the instances}
\label{subsec:instances}

Our data set consists of a subset of the symmetric traveling salesman problem instances available at TSPLIB (\url{http://comopt.ifi.uni-heidelberg.de/software/TSPLIB95/}). Among these 111 available instances, we select 38 for which the single-level OP can be solved optimally in less than 15 minutes in our environment, as our test instances.

For determining the node prizes, we consider two options: unit prizes, denoted by $u$; and pseudorandom prizes, denoted by $r$. For random prizes, we follow the approach in \cite{fischetti1998solving} and generate the values according to the equation $p_i=1+(7141i+73)\mod{100}$. Two interdiction budget levels $Q_l \in \{5,8\}$ are considered. The distance budget is determined as $B_f=0.5 \nu$ where $\nu$ denotes the optimal TSP tour length which is provided with the instances. The resulting instance set is available at \url{https://msinnl.github.io/pages/instancescodes.html}. 

\subsection{Results of the B\&C algorithm}
\label{sec:numericresults}

In our experiments with the B\&C algorithm, we consider the following algorithmic settings that are obtained by including a subset of the enhancement strategies we propose in Section \ref{sec:enhancement}: 

   \begin{itemize}
   \setlength\itemsep{0.2em}
        \item I: Only integer solutions are separated, in an exact way.
        \item IF: Both integer (exact) and fractional (heuristic) solutions are separated.
        \item IFH: Both integer (heuristic + exact) and fractional (heuristic) solutions are separated.
        \item IFHC: In addition to IFH we keep a (follower) cut pool.
        \item IFHCP: In addition to IFHC we apply \texttt{FollowerPreprocessing}.
	\end{itemize}

The lower cutoff strategy is applied in all settings by default. The follower solution pool $\mathcal{C}$ is used in all settings except I, where we do not generate heuristic follower solutions. In setting IFH, while separating integer solutions, we first try the heuristic method shown in Algorithm \ref{alg:FH}. If it does not yield a violated cut, then we solve the follower problem optimally to obtain a violated cut if there exists one (exact separation). Note that whenever we keep a (follower) cut pool, its size is bounded by 5000 cuts in total. In the experiments with fractional separation, at any (leader) B\&C node with a fractional node solution at most 10 passes of cut generation are allowed. 

In Table \ref{tab:BC_summary_unit} and Table \ref{tab:BC_summary_rand} we show some numerical results of our experiments, as averages over the instances with the same leader budget $Q_\ell$. In the columns we show the algorithmic setting, total running time in seconds ($t(s.)$), time to generate the interdiction cuts including the time to solve (SEP) ($t_{\text{SEP}}(s.)$), the optimality gap at the end of time limit (Gap(\%)), the optimality gap at the root node (rGap(\%)), the number of optimally solved instances out of all 38 (nOpt), the number of B\&C nodes generated (nBBnode), the number of interdiction cuts at integer solutions (intCuts), and the number of interdiction cuts at fractional solutions (fracCuts). The numbers are averages over 38 instances. 

\begin{table}[htbp]
  \centering
  \caption{Average results of the unit-prize instances}
    \begin{tabular}{clrrrrrrrr}
    \hline
    \multicolumn{1}{l}{$Q_\ell$} & Setting & \multicolumn{1}{l}{$t(s.)$} & \multicolumn{1}{l}{$t_{\text{SEP}}(s.)$} & \multicolumn{1}{l}{Gap(\%)} & \multicolumn{1}{l}{rGap(\%)} & \multicolumn{1}{l}{nOpt} & \multicolumn{1}{l}{nBBnode} & \multicolumn{1}{l}{intCuts} & \multicolumn{1}{l}{fracCuts} \\
    \hline
    \multirow{5}[2]{*}{5} & I     & 1746.8 & 1742.6 & 3.2   & 6.6   & 25    & 992.1 & 167.7 & 0.0 \\
          & IF    & 1378.2 & 1376.0 & 3.1   & 6.3   & 26    & 63.1  & 44.9  & 49.4 \\
          & IFH   & 747.0 & 741.3 & 0.2   & 4.6   & 33    & 484.1 & 453.6 & 150.1 \\
          & IFHC  & 277.9 & 273.8 & 0.0   & 4.4   & 38    & 647.0 & 382.8 & 180.9 \\
          & IFHCP & 270.1 & 261.2 & 0.0   & 4.3   & 38    & 546.0 & 473.0 & 155.1 \\
    \hline
    \multirow{5}[2]{*}{8} & I     & 2329.8 & 2317.4 & 2.1   & 7.5   & 16    & 5833.5 & 433.4 & 0.0 \\
          & IF    & 2005.8 & 1990.7 & 1.7   & 7.2   & 19    & 547.2 & 113.7 & 712.4 \\
          & IFH   & 1206.8 & 1024.5 & 0.6   & 9.0   & 29    & 5926.0 & 2034.3 & 2770.8 \\
          & IFHC  & 761.0 & 490.9 & 0.1   & 9.0   & 35    & 7279.9 & 2471.4 & 3306.1 \\
          & IFHCP & 766.1 & 536.2 & 0.2   & 9.1   & 34    & 5942.3 & 2548.4 & 3350.7 \\
    \hline
    \end{tabular}%
  \label{tab:BC_summary_unit}%
\end{table}%

\begin{table}[htbp]
  \centering
  \caption{Average results of the random-prize instances}
    \begin{tabular}{clrrrrrrrr}
    \hline
    \multicolumn{1}{l}{$Q_\ell$} & Setting & \multicolumn{1}{l}{$t(s.)$} & \multicolumn{1}{l}{$t_{\text{SEP}}(s.)$} & \multicolumn{1}{l}{Gap(\%)} & \multicolumn{1}{l}{rGap(\%)} & \multicolumn{1}{l}{nOpt} & \multicolumn{1}{l}{nBBnode} & \multicolumn{1}{l}{intCuts} & \multicolumn{1}{l}{fracCuts} \\
    \hline
    \multirow{5}[2]{*}{5} & I     & 2281.2 & 2277.9 & 1.1   & 5.3   & 15    & 824.0 & 112.7 & 0.0 \\
          & IF    & 1888.6 & 1885.4 & 0.7   & 4.4   & 23    & 349.4 & 39.3  & 224.3 \\
          & IFH   & 1756.4 & 1753.2 & 0.7   & 4.7   & 24    & 515.2 & 81.5  & 264.4 \\
          & IFHC  & 1259.9 & 1252.0 & 0.3   & 4.6   & 30    & 860.7 & 131.3 & 437.9 \\
          & IFHCP & 1142.9 & 1134.0 & 0.3   & 4.7   & 30    & 852.2 & 138.1 & 447.8 \\
    \hline
    \multirow{5}[2]{*}{8} & I     & 2545.9 & 2538.2 & 3.1   & 9.7   & 14    & 4477.5 & 188.2 & 0.0 \\
          & IF    & 2170.1 & 2155.3 & 2.3   & 8.1   & 18    & 1006.9 & 51.5  & 826.1 \\
          & IFH   & 2110.9 & 2095.4 & 2.3   & 8.5   & 17    & 1234.3 & 144.0 & 930.3 \\
          & IFHC  & 1900.0 & 1838.5 & 1.3   & 8.4   & 23    & 3639.3 & 224.4 & 2125.2 \\
          & IFHCP & 1773.0 & 1704.4 & 1.1   & 8.5   & 23    & 3735.7 & 245.0 & 2320.9 \\
    \hline
    \end{tabular}%
  \label{tab:BC_summary_rand}%
\end{table}%

The results indicate that separating fractional solutions is very effective in decreasing the solution time and tree size. For the unit-prize instances, separating integer solutions in a heuristic way also seems to significantly decrease the solution time, through a reduced number of times we need to solve \myref{SEP}. This is not the case for the random-prize instances, which we explain by the poor quality heuristic solutions. Solving \myref{SEP} usually cannot be avoided because of failing to find a violated cut via the heuristic, even though there exists one. Keeping a follower cut pool on the other hand, is effective under both prize choices. It reduces the overall solution time by reducing the time spent to solve \myref{SEP}. Lastly, \texttt{FollowePreprocessing} brings some improvement in terms of solution time of the random-prize instances, although its marginal contribution is not as large as of the other components. 

Figure \ref{fig:cumDist} shows the cumulative distribution of the running times of all instances under different algorithmic settings. While we are able to solve 44\% of the instances optimally in one hour under the basic setting I, this ratio is reached in two minutes under the setting IFHCP which is the best performer in terms of run time.

\begin{figure}[htbp]
\centering
\includegraphics[width=0.6\textwidth]{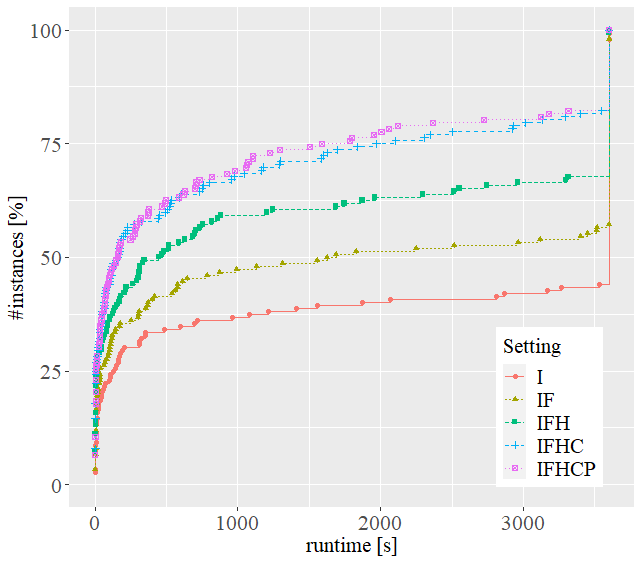}
\caption{Cumulative distribution of running times of all instances under different B\&C settings.}
\label{fig:cumDist}
\end{figure}

\subsection{Results of the genetic algorithm}

For our genetic algorithm, we consider the following parameter values, which we determined in preliminary tests:
Before the main iterations we initialize the follower solution pool $\mathcal{C}$ by applying the heuristic of \cite{fischetti1998solving} $k_0=10$ times. An initial population of $p_0=20$ individuals is created via \texttt{Greedy} using a skipping probability $p_s=0.4$. At every iteration 3-way tournament is applied to choose the parents. We allow a maximum population size of $p_{max}=100$, and limit the size of $\mathcal{C}$ by 2000 as the complexity of \texttt{EstimateObjective} increases with it. Every 10 iterations the objective value of the fittest individual is re-evaluated, which is likely to change its rank. The maximum number of iterations is $n_{maxIter}=5000$.

In Table \ref{tab:GA}, we compare the B\&C results with the results of the genetic algorithm given in Algorithm \ref{alg:GA}. The table displays the results of each instance, in terms of run time in seconds ($t(s.)$), the best objective value obtained via Algorithm \ref{alg:GA} ($z_{GA}$), the best objective value obtained via the B\&C method ($z_{BC}$), and their relative difference $\Delta = 100(z_{GA}-z_{BC})/z_{BC}$. Since not all $z_{BC}$ values are optimal objective values, $\Delta$ shows how much the heuristic solution objective value is far away from the best primal bound we have. The average solution time of the genetic algorithm is less than three minutes over all instances, whereas it finds solutions which are only 2.46\% away from the B\&C objective values, on the average. The maximum deviation from $z_{BC}$ is 10.42\% for unit-prize instances and 13.29\% for random-prize instances. These results indicate that our genetic algorithm performs well in case we need to obtain a high-quality feasible solution in short time.
%

\begin{table}[htbp]
  \centering 
  \caption{Genetic algorithm results compared to the B\&C results.}
  \scriptsize
    \begin{tabular}{ccrrrr|ccrrrr}
    \toprule
    \multicolumn{6}{c|}{Unit-prize}               & \multicolumn{6}{c}{Random-prize} \\
    \midrule
    Instance & Q     & \multicolumn{1}{l}{$t(s.)$} & \multicolumn{1}{l}{$z_{GA}$} & \multicolumn{1}{l}{$z_{BC}$} & \multicolumn{1}{c|}{$\Delta$} & \multicolumn{1}{l}{Instance} & Q     & \multicolumn{1}{l}{$t(s.)$} & \multicolumn{1}{l}{$z_{GA}$} & \multicolumn{1}{l}{$z_{BC}$} & \multicolumn{1}{c}{$\Delta$} \\
    \midrule
    \multicolumn{1}{l}{att48} & 5     & 12.84 & 26    & 26    & 0.00  & \multicolumn{1}{l}{att48} & 5     & 10.90 & 1219  & 1219  & 0.00 \\
    \multicolumn{1}{l}{att48} & 8     & 21.00 & 23    & 23    & 0.00  & \multicolumn{1}{l}{att48} & 8     & 14.59 & 1103  & 1038  & 6.26 \\
    \multicolumn{1}{l}{bayg29} & 5     & 5.23  & 12    & 12    & 0.00  & \multicolumn{1}{l}{bayg29} & 5     & 3.18  & 603   & 603   & 0.00 \\
    \multicolumn{1}{l}{bayg29} & 8     & 8.89  & 11    & 11    & 0.00  & \multicolumn{1}{l}{bayg29} & 8     & 5.60  & 476   & 474   & 0.42 \\
    \multicolumn{1}{l}{bays29} & 5     & 6.81  & 13    & 13    & 0.00  & \multicolumn{1}{l}{bays29} & 5     & 4.08  & 630   & 630   & 0.00 \\
    \multicolumn{1}{l}{bays29} & 8     & 4.63  & 11    & 11    & 0.00  & \multicolumn{1}{l}{bays29} & 8     & 6.04  & 463   & 463   & 0.00 \\
    \multicolumn{1}{l}{berlin52} & 5     & 20.61 & 31    & 31    & 0.00  & \multicolumn{1}{l}{berlin52} & 5     & 19.55 & 1583  & 1548  & 2.26 \\
    \multicolumn{1}{l}{berlin52} & 8     & 31.69 & 28    & 28    & 0.00  & \multicolumn{1}{l}{berlin52} & 8     & 24.46 & 1372  & 1318  & 4.10 \\
    \multicolumn{1}{l}{bier127} & 5     & 242.67 & 98    & 97    & 1.03  & \multicolumn{1}{l}{bier127} & 5     & 337.43 & 4892  & 4883  & 0.18 \\
    \multicolumn{1}{l}{bier127} & 8     & 250.25 & 95    & 95    & 0.00  & \multicolumn{1}{l}{bier127} & 8     & 483.29 & 4622  & 4621  & 0.02 \\
    \multicolumn{1}{l}{brazil58} & 5     & 32.22 & 40    & 40    & 0.00  & \multicolumn{1}{l}{brazil58} & 5     & 20.24 & 1909  & 1884  & 1.33 \\
    \multicolumn{1}{l}{brazil58} & 8     & 33.21 & 37    & 37    & 0.00  & \multicolumn{1}{l}{brazil58} & 8     & 30.61 & 1673  & 1664  & 0.54 \\
    \multicolumn{1}{l}{ch130} & 5     & 238.00 & 76    & 73    & 4.11  & \multicolumn{1}{l}{ch130} & 5     & 322.17 & 3969  & 3908  & 1.56 \\
    \multicolumn{1}{l}{ch130} & 8     & 294.88 & 74    & 70    & 5.71  & \multicolumn{1}{l}{ch130} & 8     & 350.66 & 3758  & 3706  & 1.40 \\
    \multicolumn{1}{l}{ch150} & 5     & 311.83 & 83    & 81    & 2.47  & \multicolumn{1}{l}{ch150} & 5     & 677.18 & 4433  & 4399  & 0.77 \\
    \multicolumn{1}{l}{ch150} & 8     & 368.94 & 83    & 78    & 6.41  & \multicolumn{1}{l}{ch150} & 8     & 797.05 & 4270  & 4197  & 1.74 \\
    \multicolumn{1}{l}{dantzig42} & 5     & 9.08  & 21    & 21    & 0.00  & \multicolumn{1}{l}{dantzig42} & 5     & 6.37  & 1005  & 1005  & 0.00 \\
    \multicolumn{1}{l}{dantzig42} & 8     & 11.93 & 19    & 18    & 5.56  & \multicolumn{1}{l}{dantzig42} & 8     & 11.17 & 817   & 802   & 1.87 \\
    \multicolumn{1}{l}{eil101} & 5     & 107.61 & 60    & 59    & 1.69  & \multicolumn{1}{l}{eil101} & 5     & 95.97 & 3160  & 3151  & 0.29 \\
    \multicolumn{1}{l}{eil101} & 8     & 140.35 & 58    & 56    & 3.57  & \multicolumn{1}{l}{eil101} & 8     & 131.83 & 2992  & 2944  & 1.63 \\
    \multicolumn{1}{l}{eil51} & 5     & 13.51 & 25    & 24    & 4.17  & \multicolumn{1}{l}{eil51} & 5     & 10.54 & 1471  & 1379  & 6.67 \\
    \multicolumn{1}{l}{eil51} & 8     & 21.61 & 24    & 23    & 4.35  & \multicolumn{1}{l}{eil51} & 8     & 18.04 & 1260  & 1214  & 3.79 \\
    \multicolumn{1}{l}{eil76} & 5     & 35.49 & 42    & 41    & 2.44  & \multicolumn{1}{l}{eil76} & 5     & 24.86 & 2211  & 2144  & 3.13 \\
    \multicolumn{1}{l}{eil76} & 8     & 43.07 & 40    & 39    & 2.56  & \multicolumn{1}{l}{eil76} & 8     & 60.23 & 2034  & 1954  & 4.09 \\
    \multicolumn{1}{l}{fri26} & 5     & 3.43  & 9     & 9     & 0.00  & \multicolumn{1}{l}{fri26} & 5     & 2.32  & 455   & 410   & 10.98 \\
    \multicolumn{1}{l}{fri26} & 8     & 4.77  & 7     & 7     & 0.00  & \multicolumn{1}{l}{fri26} & 8     & 2.16  & 341   & 301   & 13.29 \\
    \multicolumn{1}{l}{gr120} & 5     & 161.26 & 73    & 69    & 5.80  & \multicolumn{1}{l}{gr120} & 5     & 304.69 & 3861  & 3829  & 0.84 \\
    \multicolumn{1}{l}{gr120} & 8     & 220.86 & 69    & 66    & 4.55  & \multicolumn{1}{l}{gr120} & 8     & 338.55 & 3608  & 3604  & 0.11 \\
    \multicolumn{1}{l}{gr17} & 5     & 1.55  & 6     & 6     & 0.00  & \multicolumn{1}{l}{gr17} & 5     & 0.92  & 194   & 194   & 0.00 \\
    \multicolumn{1}{l}{gr17} & 8     & 1.04  & 3     & 3     & 0.00  & \multicolumn{1}{l}{gr17} & 8     & 1.18  & 118   & 118   & 0.00 \\
    \multicolumn{1}{l}{gr21} & 5     & 1.96  & 7     & 7     & 0.00  & \multicolumn{1}{l}{gr21} & 5     & 1.72  & 303   & 303   & 0.00 \\
    \multicolumn{1}{l}{gr21} & 8     & 1.80  & 5     & 5     & 0.00  & \multicolumn{1}{l}{gr21} & 8     & 1.57  & 191   & 191   & 0.00 \\
    \multicolumn{1}{l}{gr24} & 5     & 2.58  & 8     & 8     & 0.00  & \multicolumn{1}{l}{gr24} & 5     & 2.23  & 469   & 430   & 9.07 \\
    \multicolumn{1}{l}{gr24} & 8     & 2.74  & 6     & 6     & 0.00  & \multicolumn{1}{l}{gr24} & 8     & 2.04  & 304   & 304   & 0.00 \\
    \multicolumn{1}{l}{gr48} & 5     & 16.35 & 25    & 25    & 0.00  & \multicolumn{1}{l}{gr48} & 5     & 10.17 & 1139  & 1130  & 0.80 \\
    \multicolumn{1}{l}{gr48} & 8     & 22.72 & 22    & 22    & 0.00  & \multicolumn{1}{l}{gr48} & 8     & 17.55 & 985   & 985   & 0.00 \\
    \multicolumn{1}{l}{hk48} & 5     & 14.99 & 24    & 24    & 0.00  & \multicolumn{1}{l}{hk48} & 5     & 12.71 & 1174  & 1167  & 0.60 \\
    \multicolumn{1}{l}{hk48} & 8     & 23.95 & 23    & 22    & 4.55  & \multicolumn{1}{l}{hk48} & 8     & 18.35 & 1110  & 1027  & 8.08 \\
    \multicolumn{1}{l}{kroA100} & 5     & 68.38 & 54    & 51    & 5.88  & \multicolumn{1}{l}{kroA100} & 5     & 112.90 & 2838  & 2745  & 3.39 \\
    \multicolumn{1}{l}{kroA100} & 8     & 108.76 & 52    & 48    & 8.33  & \multicolumn{1}{l}{kroA100} & 8     & 107.58 & 2645  & 2521  & 4.92 \\
    \multicolumn{1}{l}{kroA150} & 5     & 289.02 & 83    & 80    & 3.75  & \multicolumn{1}{l}{kroA150} & 5     & 593.48 & 4566  & 4473  & 2.08 \\
    \multicolumn{1}{l}{kroA150} & 8     & 779.43 & 80    & 78    & 2.56  & \multicolumn{1}{l}{kroA150} & 8     & 763.48 & 4344  & 4237  & 2.53 \\
    \multicolumn{1}{l}{kroB100} & 5     & 77.17 & 54    & 52    & 3.85  & \multicolumn{1}{l}{kroB100} & 5     & 201.57 & 2782  & 2725  & 2.09 \\
    \multicolumn{1}{l}{kroB100} & 8     & 126.44 & 52    & 49    & 6.12  & \multicolumn{1}{l}{kroB100} & 8     & 169.17 & 2634  & 2559  & 2.93 \\
    \multicolumn{1}{l}{kroB150} & 5     & 443.80 & 83    & 81    & 2.47  & \multicolumn{1}{l}{kroB150} & 5     & 536.86 & 4560  & 4524  & 0.80 \\
    \multicolumn{1}{l}{kroB150} & 8     & 586.91 & 80    & 78    & 2.56  & \multicolumn{1}{l}{kroB150} & 8     & 604.61 & 4388  & 4346  & 0.97 \\
    \multicolumn{1}{l}{kroC100} & 5     & 73.05 & 51    & 50    & 2.00  & \multicolumn{1}{l}{kroC100} & 5     & 118.48 & 2667  & 2575  & 3.57 \\
    \multicolumn{1}{l}{kroC100} & 8     & 112.67 & 50    & 47    & 6.38  & \multicolumn{1}{l}{kroC100} & 8     & 193.43 & 2453  & 2411  & 1.74 \\
    \multicolumn{1}{l}{kroD100} & 5     & 66.62 & 55    & 53    & 3.77  & \multicolumn{1}{l}{kroD100} & 5     & 102.16 & 2965  & 2786  & 6.42 \\
    \multicolumn{1}{l}{kroD100} & 8     & 82.87 & 51    & 50    & 2.00  & \multicolumn{1}{l}{kroD100} & 8     & 141.12 & 2810  & 2583  & 8.79 \\
    \multicolumn{1}{l}{kroE100} & 5     & 64.75 & 53    & 51    & 3.92  & \multicolumn{1}{l}{kroE100} & 5     & 65.01 & 2934  & 2859  & 2.62 \\
    \multicolumn{1}{l}{kroE100} & 8     & 97.38 & 50    & 49    & 2.04  & \multicolumn{1}{l}{kroE100} & 8     & 102.79 & 2795  & 2648  & 5.55 \\
    \multicolumn{1}{l}{lin105} & 5     & 111.21 & 61    & 60    & 1.67  & \multicolumn{1}{l}{lin105} & 5     & 72.78 & 3147  & 3135  & 0.38 \\
    \multicolumn{1}{l}{lin105} & 8     & 143.84 & 58    & 57    & 1.75  & \multicolumn{1}{l}{lin105} & 8     & 99.07 & 2924  & 2901  & 0.79 \\
    \multicolumn{1}{l}{pr107} & 5     & 51.18 & 53    & 48    & 10.42 & \multicolumn{1}{l}{pr107} & 5     & 179.06 & 2189  & 2189  & 0.00 \\
    \multicolumn{1}{l}{pr107} & 8     & 133.05 & 45    & 45    & 0.00  & \multicolumn{1}{l}{pr107} & 8     & 180.22 & 1926  & 1926  & 0.00 \\
    \multicolumn{1}{l}{pr124} & 5     & 137.90 & 69    & 69    & 0.00  & \multicolumn{1}{l}{pr124} & 5     & 156.18 & 3390  & 3364  & 0.77 \\
    \multicolumn{1}{l}{pr124} & 8     & 138.02 & 66    & 66    & 0.00  & \multicolumn{1}{l}{pr124} & 8     & 110.76 & 3227  & 3129  & 3.13 \\
    \multicolumn{1}{l}{pr136} & 5     & 197.26 & 67    & 66    & 1.52  & \multicolumn{1}{l}{pr136} & 5     & 297.52 & 3922  & 3830  & 2.40 \\
    \multicolumn{1}{l}{pr136} & 8     & 269.29 & 67    & 64    & 4.69  & \multicolumn{1}{l}{pr136} & 8     & 349.70 & 3766  & 3693  & 1.98 \\
    \multicolumn{1}{l}{pr144} & 5     & 339.57 & 72    & 71    & 1.41  & \multicolumn{1}{l}{pr144} & 5     & 146.42 & 3781  & 3514  & 7.60 \\
    \multicolumn{1}{l}{pr144} & 8     & 259.29 & 72    & 68    & 5.88  & \multicolumn{1}{l}{pr144} & 8     & 587.87 & 3499  & 3261  & 7.30 \\
    \multicolumn{1}{l}{pr152} & 5     & 200.72 & 75    & 71    & 5.63  & \multicolumn{1}{l}{pr152} & 5     & 328.82 & 3826  & 3763  & 1.67 \\
    \multicolumn{1}{l}{pr152} & 8     & 290.43 & 71    & 69    & 2.90  & \multicolumn{1}{l}{pr152} & 8     & 385.70 & 3619  & 3539  & 2.26 \\
    \multicolumn{1}{l}{pr76} & 5     & 46.95 & 44    & 43    & 2.33  & \multicolumn{1}{l}{pr76} & 5     & 47.73 & 2254  & 2247  & 0.31 \\
    \multicolumn{1}{l}{pr76} & 8     & 52.13 & 42    & 41    & 2.44  & \multicolumn{1}{l}{pr76} & 8     & 56.53 & 2059  & 2025  & 1.68 \\
    \multicolumn{1}{l}{rat99} & 5     & 103.53 & 50    & 47    & 6.38  & \multicolumn{1}{l}{rat99} & 5     & 85.97 & 2650  & 2596  & 2.08 \\
    \multicolumn{1}{l}{rat99} & 8     & 128.07 & 48    & 45    & 6.67  & \multicolumn{1}{l}{rat99} & 8     & 144.29 & 2553  & 2438  & 4.72 \\
    \multicolumn{1}{l}{rd100} & 5     & 75.95 & 55    & 55    & 0.00  & \multicolumn{1}{l}{rd100} & 5     & 108.51 & 3086  & 3007  & 2.63 \\
    \multicolumn{1}{l}{rd100} & 8     & 101.07 & 52    & 52    & 0.00  & \multicolumn{1}{l}{rd100} & 8     & 140.61 & 2938  & 2809  & 4.59 \\
    \multicolumn{1}{l}{st70} & 5     & 38.16 & 38    & 37    & 2.70  & \multicolumn{1}{l}{st70} & 5     & 36.05 & 1924  & 1855  & 3.72 \\
    \multicolumn{1}{l}{st70} & 8     & 48.81 & 36    & 34    & 5.88  & \multicolumn{1}{l}{st70} & 8     & 44.46 & 1730  & 1672  & 3.47 \\
    \multicolumn{1}{l}{swiss42} & 5     & 11.40 & 22    & 22    & 0.00  & \multicolumn{1}{l}{swiss42} & 5     & 6.97  & 1004  & 1004  & 0.00 \\
    \multicolumn{1}{l}{swiss42} & 8     & 15.97 & 20    & 19    & 5.26  & \multicolumn{1}{l}{swiss42} & 8     & 9.80  & 810   & 810   & 0.00 \\
    \multicolumn{1}{l}{u159} & 5     & 642.95 & 87    & 87    & 0.00  & \multicolumn{1}{l}{u159} & 5     & 761.40 & 4636  & 4551  & 1.87 \\
    \multicolumn{1}{l}{u159} & 8     & 434.90 & 88    & 84    & 4.76  & \multicolumn{1}{l}{u159} & 8     & 816.76 & 4379  & 4374  & 0.11 \\
    \midrule
    \multicolumn{2}{c}{Average} & 127.57 & 46.36 & 45.04 & 2.46  & \multicolumn{2}{c}{Average} & 173.00 & 2378.91 & 2325.86 & 2.47 \\
    \bottomrule
    \end{tabular}%
  \label{tab:GA}%
\end{table}%

\section{Conclusions and future work}
\label{sec:conclu}

In this work, we propose the orienteering interdiction game where two players (leader and follower) compete in a hierarchical manner: The follower tries to maximize the total profit collected by visiting nodes (i.e., the sum of the prizes of the visited nodes) and the leader wants to minimize this amount by interdicting nodes (if a node is interdicted, the follower does not gain the prize of the node when visiting it). 
Such a setting may be encountered during political campaign planning to damage the effectiveness of the competitor's canvassing, 
in security operations where the routing-based activities of an adversary agent is prevented, or in the analysis of worst-case scenarios of the attacks towards patrolling security forces.  
This zero-sum Stackelberg game can be modeled as a bilevel optimization problem and further reformulated as a single-level problem adapting the so-called \emph{interdiction cuts} which were introduced in \citet{fischetti2019interdiction} for interdiction games fulfilling a certain monotonicity assumption. We propose such a reformulation and develop a branch-and-cut algorithm to solve it exactly. In addition, we develop a genetic algorithm in which the fitness value of an individual is estimated heuristically using a solution pool. We conduct a computational study by creating instances of \OIG using a set of TSP instances in the literature. The results show that the performance of the branch-and-cut method can be drastically improved by means of the proposed enhancement strategies. The genetic algorithm yields solutions that are similar to the B\&C solutions in terms of the objective function value, in reasonable time.   

There are various avenues for further work: It could be interesting to try to design other exact solution algorithms, which do not use a reformulation based on interdiction cuts. Moreover, the development of other heuristic algorithms could also be a fruitful direction for further work. In particular, the fact that obtaining the objective function value of any feasible solution requires the solution of the NP-hard orienteering problem presents an intriguing challenge in this context. 
Moreover, the \OIG could be extended by adding (topological) constraints to the leader problem, e.g., it could be imposed that the leader also needs to solve an orienteering problem, and the nodes visited by the leader in her or his tour are the nodes which are then interdicted for the follower. Finally, it could also be interesting to consider other orienteering problems, such as the team orienteering problem, the orienteering problem with time windows, and the multi-period orienteering problem in a similar game-theoretic setting.

\paragraph{Acknowledgments}  E. \'Alvarez-Miranda  acknowledges the support of the National Agency of Research
and Development (ANID), Chile, through the grant FONDECYT N.1180670 and through the Complex Engineering Systems Institute ANID PIA/BASAL AFB180003. 

\bibliographystyle{elsarticle-harv}
\bibliography{bibOrienteering,bibInterdiction,bibCanvassing,bibPatrolling}

\begin{thebibliography}{52}
\expandafter\ifx\csname natexlab\endcsname\relax\def\natexlab#1{#1}\fi
\providecommand{\url}[1]{\texttt{#1}}
\providecommand{\href}[2]{#2}
\providecommand{\path}[1]{#1}
\providecommand{\DOIprefix}{doi:}
\providecommand{\ArXivprefix}{arXiv:}
\providecommand{\URLprefix}{URL: }
\providecommand{\Pubmedprefix}{pmid:}
\providecommand{\doi}[1]{\href{http://dx.doi.org/#1}{\path{#1}}}
\providecommand{\Pubmed}[1]{\href{pmid:#1}{\path{#1}}}
\providecommand{\bibinfo}[2]{#2}
\ifx\xfnm\relax \def\xfnm[#1]{\unskip,\space#1}\fi
\bibitem[{Beck et~al.(2023)Beck, Ljubi{\'c} and Schmidt}]{beck2023survey}
\bibinfo{author}{Beck, Y.}, \bibinfo{author}{Ljubi{\'c}, I.},
  \bibinfo{author}{Schmidt, M.}, \bibinfo{year}{2023}.
\newblock \bibinfo{title}{A survey on bilevel optimization under uncertainty}.
\newblock \bibinfo{journal}{European Journal of Operational Research}
  \bibinfo{volume}{311}, \bibinfo{pages}{401--426}.
\bibitem[{Bhatti et~al.(2019)Bhatti, Dahlgaard, Hansen and
  Hansen}]{bhatti2019door}
\bibinfo{author}{Bhatti, Y.}, \bibinfo{author}{Dahlgaard, J.},
  \bibinfo{author}{Hansen, J.}, \bibinfo{author}{Hansen, K.},
  \bibinfo{year}{2019}.
\newblock \bibinfo{title}{Is door-to-door canvassing effective in europe?
  evidence from a meta-study across six european countries}.
\newblock \bibinfo{journal}{British Journal of Political Science}
  \bibinfo{volume}{49}, \bibinfo{pages}{279--290}.
\bibitem[{Bidgoli and Kheirkhah(2018)}]{bidgoli2018arc}
\bibinfo{author}{Bidgoli, M.M.}, \bibinfo{author}{Kheirkhah, A.},
  \bibinfo{year}{2018}.
\newblock \bibinfo{title}{An arc interdiction vehicle routing problem with
  information asymmetry}.
\newblock \bibinfo{journal}{Computers \& Industrial Engineering}
  \bibinfo{volume}{115}, \bibinfo{pages}{520--531}.
\bibitem[{Brown et~al.(2006)Brown, Carlyle, Salmer{\'o}n and
  Wood}]{brown2006defending}
\bibinfo{author}{Brown, G.}, \bibinfo{author}{Carlyle, M.},
  \bibinfo{author}{Salmer{\'o}n, J.}, \bibinfo{author}{Wood, K.},
  \bibinfo{year}{2006}.
\newblock \bibinfo{title}{Defending critical infrastructure}.
\newblock \bibinfo{journal}{Interfaces} \bibinfo{volume}{36},
  \bibinfo{pages}{530--544}.
\bibitem[{Calvete et~al.(2011)Calvete, Gal{\'e} and
  Oliveros}]{calvete2011bilevel}
\bibinfo{author}{Calvete, H.I.}, \bibinfo{author}{Gal{\'e}, C.},
  \bibinfo{author}{Oliveros, M.J.}, \bibinfo{year}{2011}.
\newblock \bibinfo{title}{Bilevel model for production--distribution planning
  solved by using ant colony optimization}.
\newblock \bibinfo{journal}{Computers \& Operations Research}
  \bibinfo{volume}{38}, \bibinfo{pages}{320--327}.
\bibitem[{Camacho-Vallejo et~al.(2022)Camacho-Vallejo, L{\'o}pez-Vera, Smith
  and Gonz{\'a}lez-Velarde}]{camacho2022tabu}
\bibinfo{author}{Camacho-Vallejo, J.F.}, \bibinfo{author}{L{\'o}pez-Vera, L.},
  \bibinfo{author}{Smith, A.E.}, \bibinfo{author}{Gonz{\'a}lez-Velarde, J.L.},
  \bibinfo{year}{2022}.
\newblock \bibinfo{title}{A tabu search algorithm to solve a green logistics
  bi-objective bi-level problem}.
\newblock \bibinfo{journal}{Annals of Operations Research}
  \bibinfo{volume}{316}, \bibinfo{pages}{927--953}.
\bibitem[{Cerulli et~al.(2023)Cerulli, Archetti, Fernandez and
  Ljubic}]{cerulli2023bilevel}
\bibinfo{author}{Cerulli, M.}, \bibinfo{author}{Archetti, C.},
  \bibinfo{author}{Fernandez, E.}, \bibinfo{author}{Ljubic, I.},
  \bibinfo{year}{2023}.
\newblock \bibinfo{title}{A bilevel approach for compensation and routing
  decisions in last-mile delivery}.
\newblock \bibinfo{journal}{arXiv preprint arXiv:2304.09170} .
\bibitem[{Chao et~al.(1996)Chao, Golden and Wasil}]{chao1996team}
\bibinfo{author}{Chao, I.M.}, \bibinfo{author}{Golden, B.L.},
  \bibinfo{author}{Wasil, E.A.}, \bibinfo{year}{1996}.
\newblock \bibinfo{title}{The team orienteering problem}.
\newblock \bibinfo{journal}{European Journal of Operational Research}
  \bibinfo{volume}{88}, \bibinfo{pages}{464--474}.
\bibitem[{Church et~al.(2004)Church, Scaparra and
  Middleton}]{church2004identifying}
\bibinfo{author}{Church, R.L.}, \bibinfo{author}{Scaparra, M.P.},
  \bibinfo{author}{Middleton, R.S.}, \bibinfo{year}{2004}.
\newblock \bibinfo{title}{Identifying critical infrastructure: the median and
  covering facility interdiction problems}.
\newblock \bibinfo{journal}{Annals of the Association of American Geographers}
  \bibinfo{volume}{94}, \bibinfo{pages}{491--502}.
\bibitem[{Cruz-Garc{\'\i}a et~al.(2021)Cruz-Garc{\'\i}a,
  Mart{\'\i}nez-Far{\'\i}as, Santill{\'a}n-Hern{\'a}ndez and
  Rangel}]{cruz2021mathematical}
\bibinfo{author}{Cruz-Garc{\'\i}a, S.},
  \bibinfo{author}{Mart{\'\i}nez-Far{\'\i}as, F.},
  \bibinfo{author}{Santill{\'a}n-Hern{\'a}ndez, A.}, \bibinfo{author}{Rangel,
  E.}, \bibinfo{year}{2021}.
\newblock \bibinfo{title}{{Mathematical home burglary model with stochastic
  long crime trips and patrolling: Applied to Mexico City}}.
\newblock \bibinfo{journal}{Applied Mathematics and Computation}
  \bibinfo{volume}{396}, \bibinfo{pages}{125865}.
\bibitem[{Dantzig et~al.(1954)Dantzig, Fulkerson and
  Johnson}]{dantzig1954solution}
\bibinfo{author}{Dantzig, G.}, \bibinfo{author}{Fulkerson, R.},
  \bibinfo{author}{Johnson, S.}, \bibinfo{year}{1954}.
\newblock \bibinfo{title}{Solution of a large-scale traveling-salesman
  problem}.
\newblock \bibinfo{journal}{Journal of the Operations Research Society of
  America} \bibinfo{volume}{2}, \bibinfo{pages}{393--410}.
\bibitem[{Dantzig(1957)}]{dantzig1957discrete}
\bibinfo{author}{Dantzig, G.B.}, \bibinfo{year}{1957}.
\newblock \bibinfo{title}{Discrete-variable extremum problems}.
\newblock \bibinfo{journal}{Operations Research} \bibinfo{volume}{5},
  \bibinfo{pages}{266--288}.
\bibitem[{Dempe and Zemkoho(2020)}]{dempe2020bilevel}
\bibinfo{author}{Dempe, S.}, \bibinfo{author}{Zemkoho, A.},
  \bibinfo{year}{2020}.
\newblock \bibinfo{title}{Bilevel optimization}.
\newblock \bibinfo{publisher}{Springer}.
\bibitem[{DeNegre(2011)}]{denegre2011interdiction}
\bibinfo{author}{DeNegre, S.}, \bibinfo{year}{2011}.
\newblock \bibinfo{title}{Interdiction and discrete bilevel linear
  programming}.
\newblock \bibinfo{publisher}{Lehigh University PhD thesis}.
\bibitem[{Fischetti et~al.(2017)Fischetti, Ljubi{\'c}, Monaci and
  Sinnl}]{fischetti2017new}
\bibinfo{author}{Fischetti, M.}, \bibinfo{author}{Ljubi{\'c}, I.},
  \bibinfo{author}{Monaci, M.}, \bibinfo{author}{Sinnl, M.},
  \bibinfo{year}{2017}.
\newblock \bibinfo{title}{A new general-purpose algorithm for mixed-integer
  bilevel linear programs}.
\newblock \bibinfo{journal}{Operations Research} \bibinfo{volume}{65},
  \bibinfo{pages}{1615--1637}.
\bibitem[{Fischetti et~al.(2019)Fischetti, Ljubi{\'c}, Monaci and
  Sinnl}]{fischetti2019interdiction}
\bibinfo{author}{Fischetti, M.}, \bibinfo{author}{Ljubi{\'c}, I.},
  \bibinfo{author}{Monaci, M.}, \bibinfo{author}{Sinnl, M.},
  \bibinfo{year}{2019}.
\newblock \bibinfo{title}{Interdiction games and monotonicity, with application
  to knapsack problems}.
\newblock \bibinfo{journal}{INFORMS Journal on Computing} \bibinfo{volume}{31},
  \bibinfo{pages}{390--410}.
\bibitem[{Fischetti et~al.(1998)Fischetti, Salazar-Gonzalez and
  Toth}]{fischetti1998solving}
\bibinfo{author}{Fischetti, M.}, \bibinfo{author}{Salazar-Gonzalez, J.},
  \bibinfo{author}{Toth, P.}, \bibinfo{year}{1998}.
\newblock \bibinfo{title}{Solving the orienteering problem through
  branch-and-cut}.
\newblock \bibinfo{journal}{INFORMS Journal on Computing} \bibinfo{volume}{10},
  \bibinfo{pages}{133--148}.
\bibitem[{Furini et~al.(2019)Furini, Ljubi{\'c}, Martin and
  San~Segundo}]{furini2019maximum}
\bibinfo{author}{Furini, F.}, \bibinfo{author}{Ljubi{\'c}, I.},
  \bibinfo{author}{Martin, S.}, \bibinfo{author}{San~Segundo, P.},
  \bibinfo{year}{2019}.
\newblock \bibinfo{title}{The maximum clique interdiction problem}.
\newblock \bibinfo{journal}{European Journal of Operational Research}
  \bibinfo{volume}{277}, \bibinfo{pages}{112--127}.
\bibitem[{Golden et~al.(1987)Golden, Levy and Vohra}]{golden1987orienteering}
\bibinfo{author}{Golden, B.}, \bibinfo{author}{Levy, L.},
  \bibinfo{author}{Vohra, R.}, \bibinfo{year}{1987}.
\newblock \bibinfo{title}{The orienteering problem}.
\newblock \bibinfo{journal}{Naval Research Logistics} \bibinfo{volume}{34},
  \bibinfo{pages}{307--318}.
\bibitem[{Gunawan et~al.(2016)Gunawan, Lau and
  Vansteenwegen}]{gunawan2016orienteering}
\bibinfo{author}{Gunawan, A.}, \bibinfo{author}{Lau, H.C.},
  \bibinfo{author}{Vansteenwegen, P.}, \bibinfo{year}{2016}.
\newblock \bibinfo{title}{Orienteering problem: A survey of recent variants,
  solution approaches and applications}.
\newblock \bibinfo{journal}{European Journal of Operational Research}
  \bibinfo{volume}{255}, \bibinfo{pages}{315--332}.
\bibitem[{Ilhan et~al.(2008)Ilhan, Iravani and Daskin}]{ilhan2008orienteering}
\bibinfo{author}{Ilhan, T.}, \bibinfo{author}{Iravani, S.M.},
  \bibinfo{author}{Daskin, M.S.}, \bibinfo{year}{2008}.
\newblock \bibinfo{title}{The orienteering problem with stochastic profits}.
\newblock \bibinfo{journal}{IIE Transactions} \bibinfo{volume}{40},
  \bibinfo{pages}{406--421}.
\bibitem[{Israeli and Wood(2002)}]{israeli2002shortest}
\bibinfo{author}{Israeli, E.}, \bibinfo{author}{Wood, R.K.},
  \bibinfo{year}{2002}.
\newblock \bibinfo{title}{Shortest-path network interdiction}.
\newblock \bibinfo{journal}{Networks: An International Journal}
  \bibinfo{volume}{40}, \bibinfo{pages}{97--111}.
\bibitem[{Jia et~al.(2021)Jia, Mei and Zhang}]{jia2021bilevel}
\bibinfo{author}{Jia, Y.H.}, \bibinfo{author}{Mei, Y.}, \bibinfo{author}{Zhang,
  M.}, \bibinfo{year}{2021}.
\newblock \bibinfo{title}{A bilevel ant colony optimization algorithm for
  capacitated electric vehicle routing problem}.
\newblock \bibinfo{journal}{IEEE Transactions on Cybernetics}
  \bibinfo{volume}{52}, \bibinfo{pages}{10855--10868}.
\bibitem[{Keskin et~al.(2012)Keskin, Li, Steil and
  Spiller}]{keskin2012analysis}
\bibinfo{author}{Keskin, B.B.}, \bibinfo{author}{Li, S.R.},
  \bibinfo{author}{Steil, D.}, \bibinfo{author}{Spiller, S.},
  \bibinfo{year}{2012}.
\newblock \bibinfo{title}{Analysis of an integrated maximum covering and patrol
  routing problem}.
\newblock \bibinfo{journal}{Transportation Research Part E: Logistics and
  Transportation Review} \bibinfo{volume}{48}, \bibinfo{pages}{215--232}.
\bibitem[{Kheirkhah et~al.(2016a)Kheirkhah, Navidi and
  Bidgoli}]{kheirkhah2016improved}
\bibinfo{author}{Kheirkhah, A.}, \bibinfo{author}{Navidi, H.},
  \bibinfo{author}{Bidgoli, M.M.}, \bibinfo{year}{2016}a.
\newblock \bibinfo{title}{An improved benders decomposition algorithm for an
  arc interdiction vehicle routing problem}.
\newblock \bibinfo{journal}{IEEE Transactions on Engineering Management}
  \bibinfo{volume}{63}, \bibinfo{pages}{259--273}.
\bibitem[{Kheirkhah et~al.(2016b)Kheirkhah, Navidi and
  Messi~Bidgoli}]{kheirkhah2016bi}
\bibinfo{author}{Kheirkhah, A.}, \bibinfo{author}{Navidi, H.},
  \bibinfo{author}{Messi~Bidgoli, M.}, \bibinfo{year}{2016}b.
\newblock \bibinfo{title}{A bi-level network interdiction model for solving the
  hazmat routing problem}.
\newblock \bibinfo{journal}{International Journal of Production Research}
  \bibinfo{volume}{54}, \bibinfo{pages}{459--471}.
\bibitem[{Kleinert et~al.(2021)Kleinert, Labb{\'e}, Ljubi{\'c} and
  Schmidt}]{kleinert2021survey}
\bibinfo{author}{Kleinert, T.}, \bibinfo{author}{Labb{\'e}, M.},
  \bibinfo{author}{Ljubi{\'c}, I.}, \bibinfo{author}{Schmidt, M.},
  \bibinfo{year}{2021}.
\newblock \bibinfo{title}{A survey on mixed-integer programming techniques in
  bilevel optimization}.
\newblock \bibinfo{journal}{EURO Journal on Computational Optimization}
  \bibinfo{volume}{9}, \bibinfo{pages}{100007}.
\bibitem[{Labadie et~al.(2012)Labadie, Mansini, Melechovsk{\`y} and
  Calvo}]{labadie2012team}
\bibinfo{author}{Labadie, N.}, \bibinfo{author}{Mansini, R.},
  \bibinfo{author}{Melechovsk{\`y}, J.}, \bibinfo{author}{Calvo, R.W.},
  \bibinfo{year}{2012}.
\newblock \bibinfo{title}{The team orienteering problem with time windows: An
  lp-based granular variable neighborhood search}.
\newblock \bibinfo{journal}{European Journal of Operational Research}
  \bibinfo{volume}{220}, \bibinfo{pages}{15--27}.
\bibitem[{Laporte and Martello(1990)}]{laporte1990selective}
\bibinfo{author}{Laporte, G.}, \bibinfo{author}{Martello, S.},
  \bibinfo{year}{1990}.
\newblock \bibinfo{title}{The selective travelling salesman problem}.
\newblock \bibinfo{journal}{Discrete Applied Mathematics} \bibinfo{volume}{26},
  \bibinfo{pages}{193--207}.
\bibitem[{Lim and Smith(2007)}]{lim2007algorithms}
\bibinfo{author}{Lim, C.}, \bibinfo{author}{Smith, J.C.}, \bibinfo{year}{2007}.
\newblock \bibinfo{title}{Algorithms for discrete and continuous multicommodity
  flow network interdiction problems}.
\newblock \bibinfo{journal}{IIE Transactions} \bibinfo{volume}{39},
  \bibinfo{pages}{15--26}.
\bibitem[{Lozano et~al.(2017)Lozano, Smith and Kurz}]{lozano2017solving}
\bibinfo{author}{Lozano, L.}, \bibinfo{author}{Smith, J.},
  \bibinfo{author}{Kurz, M.}, \bibinfo{year}{2017}.
\newblock \bibinfo{title}{Solving the traveling salesman problem with
  interdiction and fortification}.
\newblock \bibinfo{journal}{Operations Research Letters} \bibinfo{volume}{45},
  \bibinfo{pages}{210--216}.
\bibitem[{Lozano and Smith(2017)}]{lozano2017backward}
\bibinfo{author}{Lozano, L.}, \bibinfo{author}{Smith, J.C.},
  \bibinfo{year}{2017}.
\newblock \bibinfo{title}{A backward sampling framework for interdiction
  problems with fortification}.
\newblock \bibinfo{journal}{INFORMS Journal on Computing} \bibinfo{volume}{29},
  \bibinfo{pages}{123--139}.
\bibitem[{Lupfer and Price(1972)}]{lupfer1972merits}
\bibinfo{author}{Lupfer, M.}, \bibinfo{author}{Price, D.},
  \bibinfo{year}{1972}.
\newblock \bibinfo{title}{On the merits of face-to-face campaigning}.
\newblock \bibinfo{journal}{Social Science Quarterly} ,
  \bibinfo{pages}{534--543}.
\bibitem[{Marinakis and Marinaki(2008)}]{marinakis2008bilevel}
\bibinfo{author}{Marinakis, Y.}, \bibinfo{author}{Marinaki, M.},
  \bibinfo{year}{2008}.
\newblock \bibinfo{title}{A bilevel genetic algorithm for a real life location
  routing problem}.
\newblock \bibinfo{journal}{International Journal of Logistics: Research and
  Applications} \bibinfo{volume}{11}, \bibinfo{pages}{49--65}.
\bibitem[{Marinakis et~al.(2007)Marinakis, Migdalas and
  Pardalos}]{marinakis2007new}
\bibinfo{author}{Marinakis, Y.}, \bibinfo{author}{Migdalas, A.},
  \bibinfo{author}{Pardalos, P.M.}, \bibinfo{year}{2007}.
\newblock \bibinfo{title}{A new bilevel formulation for the vehicle routing
  problem and a solution method using a genetic algorithm}.
\newblock \bibinfo{journal}{Journal of Global Optimization}
  \bibinfo{volume}{38}, \bibinfo{pages}{555--580}.
\bibitem[{Nadizadeh and Sabzevari~Zadeh(2021)}]{nadizadeh2021bi}
\bibinfo{author}{Nadizadeh, A.}, \bibinfo{author}{Sabzevari~Zadeh, A.},
  \bibinfo{year}{2021}.
\newblock \bibinfo{title}{A bi-level model and memetic algorithm for arc
  interdiction location-routing problem}.
\newblock \bibinfo{journal}{Computational and Applied Mathematics}
  \bibinfo{volume}{40}, \bibinfo{pages}{1--44}.
\bibitem[{Nikolakopoulos(2015)}]{nikolakopoulos2015metaheuristic}
\bibinfo{author}{Nikolakopoulos, A.}, \bibinfo{year}{2015}.
\newblock \bibinfo{title}{A metaheuristic reconstruction algorithm for solving
  bi-level vehicle routing problems with backhauls for army rapid fielding}.
\newblock \bibinfo{journal}{Military Logistics: Research Advances and Future
  Trends} , \bibinfo{pages}{141--157}.
\bibitem[{Ning and Su(2017)}]{ning2017multilevel}
\bibinfo{author}{Ning, Y.}, \bibinfo{author}{Su, T.}, \bibinfo{year}{2017}.
\newblock \bibinfo{title}{A multilevel approach for modelling vehicle routing
  problem with uncertain travelling time}.
\newblock \bibinfo{journal}{Journal of Intelligent Manufacturing}
  \bibinfo{volume}{28}, \bibinfo{pages}{683--688}.
\bibitem[{Nyman(2017)}]{nyman2017door}
\bibinfo{author}{Nyman, P.}, \bibinfo{year}{2017}.
\newblock \bibinfo{title}{{Door-to-door canvassing in the European elections:
  Evidence from a Swedish field experiment}}.
\newblock \bibinfo{journal}{Electoral Studies} \bibinfo{volume}{45},
  \bibinfo{pages}{110--118}.
\bibitem[{Parvasi et~al.(2019)Parvasi, Tavakkoli-Moghaddam, Taleizadeh and
  Soveizy}]{parvasi2019bi}
\bibinfo{author}{Parvasi, S.P.}, \bibinfo{author}{Tavakkoli-Moghaddam, R.},
  \bibinfo{author}{Taleizadeh, A.A.}, \bibinfo{author}{Soveizy, M.},
  \bibinfo{year}{2019}.
\newblock \bibinfo{title}{A bi-level bi-objective mathematical model for stop
  location in a school bus routing problem}.
\newblock \bibinfo{journal}{IFAC-PapersOnLine} \bibinfo{volume}{52},
  \bibinfo{pages}{1120--1125}.
\bibitem[{Sadati et~al.(2020a)Sadati, Aksen and Aras}]{sadati2020r}
\bibinfo{author}{Sadati, M.E.H.}, \bibinfo{author}{Aksen, D.},
  \bibinfo{author}{Aras, N.}, \bibinfo{year}{2020}a.
\newblock \bibinfo{title}{The r-interdiction selective multi-depot vehicle
  routing problem}.
\newblock \bibinfo{journal}{International Transactions in Operational Research}
  \bibinfo{volume}{27}, \bibinfo{pages}{835--866}.
\bibitem[{Sadati et~al.(2020b)Sadati, Aksen and Aras}]{sadati2020trilevel}
\bibinfo{author}{Sadati, M.E.H.}, \bibinfo{author}{Aksen, D.},
  \bibinfo{author}{Aras, N.}, \bibinfo{year}{2020}b.
\newblock \bibinfo{title}{A trilevel r-interdiction selective multi-depot
  vehicle routing problem with depot protection}.
\newblock \bibinfo{journal}{Computers \& Operations Research}
  \bibinfo{volume}{123}, \bibinfo{pages}{104996}.
\bibitem[{Sefair et~al.(2017)Sefair, Smith, Acevedo and
  Fletcher~Jr}]{sefair2017defender}
\bibinfo{author}{Sefair, J.A.}, \bibinfo{author}{Smith, J.C.},
  \bibinfo{author}{Acevedo, M.A.}, \bibinfo{author}{Fletcher~Jr, R.J.},
  \bibinfo{year}{2017}.
\newblock \bibinfo{title}{A defender-attacker model and algorithm for
  maximizing weighted expected hitting time with application to conservation
  planning}.
\newblock \bibinfo{journal}{IISE Transactions} \bibinfo{volume}{49},
  \bibinfo{pages}{1112--1128}.
\bibitem[{Smith and Song(2020)}]{smith2020survey}
\bibinfo{author}{Smith, J.C.}, \bibinfo{author}{Song, Y.},
  \bibinfo{year}{2020}.
\newblock \bibinfo{title}{A survey of network interdiction models and
  algorithms}.
\newblock \bibinfo{journal}{European Journal of Operational Research}
  \bibinfo{volume}{283}, \bibinfo{pages}{797--811}.
\bibitem[{Tahernejad et~al.(2020)Tahernejad, Ralphs and
  DeNegre}]{tahernejad2020branch}
\bibinfo{author}{Tahernejad, S.}, \bibinfo{author}{Ralphs, T.K.},
  \bibinfo{author}{DeNegre, S.T.}, \bibinfo{year}{2020}.
\newblock \bibinfo{title}{A branch-and-cut algorithm for mixed integer bilevel
  linear optimization problems and its implementation}.
\newblock \bibinfo{journal}{Mathematical Programming Computation}
  \bibinfo{volume}{12}, \bibinfo{pages}{529--568}.
\bibitem[{Tang et~al.(2016)Tang, Richard and Smith}]{tang2016class}
\bibinfo{author}{Tang, Y.}, \bibinfo{author}{Richard, J.P.P.},
  \bibinfo{author}{Smith, J.C.}, \bibinfo{year}{2016}.
\newblock \bibinfo{title}{A class of algorithms for mixed-integer bilevel
  min--max optimization}.
\newblock \bibinfo{journal}{Journal of Global Optimization}
  \bibinfo{volume}{66}, \bibinfo{pages}{225--262}.
\bibitem[{Tan{\i}nm{\i}{\c{s}} et~al.(2022)Tan{\i}nm{\i}{\c{s}}, Aras and
  Alt{\i}nel}]{taninmics2022improved}
\bibinfo{author}{Tan{\i}nm{\i}{\c{s}}, K.}, \bibinfo{author}{Aras, N.},
  \bibinfo{author}{Alt{\i}nel, {\.I}.K.}, \bibinfo{year}{2022}.
\newblock \bibinfo{title}{Improved x-space algorithm for min-max bilevel
  problems with an application to misinformation spread in social networks}.
\newblock \bibinfo{journal}{European Journal of Operational Research}
  \bibinfo{volume}{297}, \bibinfo{pages}{40--52}.
\bibitem[{Vansteenwegen et~al.(2011)Vansteenwegen, Souffriau and
  Van~Oudheusden}]{vansteenwegen2011orienteering}
\bibinfo{author}{Vansteenwegen, P.}, \bibinfo{author}{Souffriau, W.},
  \bibinfo{author}{Van~Oudheusden, D.}, \bibinfo{year}{2011}.
\newblock \bibinfo{title}{The orienteering problem: a survey}.
\newblock \bibinfo{journal}{European Journal of Operational Research}
  \bibinfo{volume}{209}, \bibinfo{pages}{1--10}.
\bibitem[{Von~Stackelberg(1952)}]{von1952theory}
\bibinfo{author}{Von~Stackelberg, H.}, \bibinfo{year}{1952}.
\newblock \bibinfo{title}{The theory of the market economy}.
\newblock \bibinfo{publisher}{Oxford University Press}.
\bibitem[{Wang and Xu(2017)}]{wang2017watermelon}
\bibinfo{author}{Wang, L.}, \bibinfo{author}{Xu, P.}, \bibinfo{year}{2017}.
\newblock \bibinfo{title}{The watermelon algorithm for the bilevel integer
  linear programming problem}.
\newblock \bibinfo{journal}{SIAM Journal on Optimization} \bibinfo{volume}{27},
  \bibinfo{pages}{1403--1430}.
\bibitem[{Wollmer(1964)}]{wollmer1964removing}
\bibinfo{author}{Wollmer, R.}, \bibinfo{year}{1964}.
\newblock \bibinfo{title}{Removing arcs from a network}.
\newblock \bibinfo{journal}{Operations Research} \bibinfo{volume}{12},
  \bibinfo{pages}{934--940}.
\bibitem[{Wood(1993)}]{wood1993deterministic}
\bibinfo{author}{Wood, R.K.}, \bibinfo{year}{1993}.
\newblock \bibinfo{title}{Deterministic network interdiction}.
\newblock \bibinfo{journal}{Mathematical and Computer Modelling}
  \bibinfo{volume}{17}, \bibinfo{pages}{1--18}.

\end{thebibliography}

\end{document}